\documentclass{amsart}

\usepackage{graphicx}

\usepackage{amssymb,amsxtra,latexsym}

\usepackage[full]{textcomp}
\usepackage[T1]{fontenc}
\usepackage[osf]{newtxtext}
\usepackage[bigdelims]{newtxmath}
\usepackage[scaled=0.94,scr=rsfs]{mathalfa}
\usepackage[scaled=0.93]{cabin}
\usepackage[varqu,varl]{inconsolata}

\usepackage{booktabs}

\usepackage{bm} 

\usepackage{tikz-cd}
\tikzcdset{arrow style=math font}
\tikzcdset{diagrams={nodes={inner sep=3pt}}}

\usepackage[pagebackref]{hyperref}

\definecolor{ceruleanblue}{rgb}{0.16, 0.32, 0.75}
\hypersetup{colorlinks=true,allcolors=ceruleanblue}

\renewcommand\theenumi{\roman{enumi}}
\renewcommand\labelenumi{(\theenumi)}

\newcommand\note[1]%
{$^\dag$\marginpar{\footnotesize{$^\dag${#1}}}}

\divide\count253 by 60 
\divide\count255 by 60
\multiply\count255 by 60 
\advance\count254 by -\count255 

\def\today{\number\year-\ifnum\month<10
0\fi\number\month-\ifnum\day<10 0\fi\number\day}

\def\hour{\ifnum\count253<10
0\number\count253\else\number\count253\fi}

\def\minute{\ifnum\count254<10
0\number\count254\else\number\count254\fi}

\newcounter{num}
{\begin{list}{\hskip\labelsep\arabic{num}.}%
{\usecounter{num}%
\setlength\leftmargin{0em}\setlength\topsep{0em}%
\setlength\parsep{0em}\setlength\partopsep{0em}%
\setlength\itemsep{0em}\setlength\labelwidth{0em}}}%
{\end{list}}

\numberwithin{equation}{subsection}

\swapnumbers

\newtheorem{theorem}[equation]{Theorem}
\newtheorem{lemma}[equation]{Lemma}
\newtheorem{proposition}[equation]{Proposition}
\newtheorem{corollary}[equation]{Corollary}

\theoremstyle{definition}

\newtheorem{remark}[equation]{Remark}

\newtheorem*{notation*}{Notation}

\newcommand\lie{\mathfrak}

\newcommand\g{\lie{g}}
\newcommand\h{\lie{h}}
\newcommand\liek{\lie{k}}

\newcommand\X{\lie{X}}

\newcommand\bb[1]{{\text{\bf#1}}}

\newcommand\Z{\bb{Z}} 

\newcommand\R{\bb{R}} 
\newcommand\C{\bb{C}}

\newcommand\M{\bb{M}}
\newcommand\T{\bb{T}}
\newcommand\V{\bb{V}}
\newcommand\W{\bb{W}}

\newcommand\ca{\mathscr}

\DeclareMathOperator\Ad{Ad}
\DeclareMathOperator\ad{ad}

\DeclareMathOperator\End{End}
\DeclareMathOperator\ev{ev}

\DeclareMathOperator\Hom{Hom}

\DeclareMathOperator\id{id}

\DeclareMathOperator\Pf{Pf}
\DeclareMathOperator\pr{pr}

\DeclareMathOperator\supp{supp}

\newcommand\group[1]{{\text{\bf#1}}}

\newcommand\SO{\group{SO}}
\newcommand\GL{\group{GL}}

\newcommand\A{\group{A}}

\newcommand\D{\group{D}}
\newcommand\E{\group{E}}
\newcommand\FF{\group{F}}

\newcommand\abs[1]{\lvert#1\rvert}

\newcommand\inner[1]{\langle#1\rangle}

\newcommand\qu[1][\kern.3ex]{/\kern-.7ex/_{\kern-.4ex#1}}
\newcommand\bigqu[1][\,\,]{\big/\kern-.85ex\big/_{\!\!#1}}

\newcommand\powl{[\kern-.3ex[}
\newcommand\powr{]\kern-.3ex]}
\newcommand\bigpowl{\bigl[\kern-.6ex\bigl[}
\newcommand\bigpowr{\bigr]\kern-.6ex\bigr]}

\newcommand\inj{\hookrightarrow}
\newcommand\sur{\mathrel{\to\kern-1.8ex\to}}
\newcommand\iso{\mathrel{\hookrightarrow\kern-1.8ex\to}}

\newcommand\longto{\longrightarrow}

\newcommand\longhookrightarrow{\lhook\joinrel\longrightarrow}

\newcommand\longinj{\longhookrightarrow}
\newcommand\longsur{\mathrel{\longrightarrow\kern-1.8ex\to}}

\newcommand\Thom{\group{Th}}
\newcommand\Eu{\group{Eu}}

\newcommand\zerodots%
{\smash{\makebox[0pt]{$_{\lower8pt\hbox{\Large$0$}}%
\ddots^{\lower3pt\hbox{\Large$0$}}$}}}
\newcommand\bigzerodots%
{\smash{\makebox[0pt]{$_{\lower6pt\hbox{\Large$0$}}%
\ddots^{\hbox{\Large$0$}}$}}}

\newcommand\F{{\ca{F}}}
\newcommand\barX{\lie{X}\kern-0.55em\overline{\phantom I}\kern0.15em}
\newcommand\barF{\F\kern-0.6em\overline{\phantom{\rm F}}}

\newcommand\cc{{\mathrm{c}}}
\newcommand\cv{{\mathrm{cv}}}

\newcommand\hor{\text{\rm hor}}

\newcommand\hhor[1]{\text{\rm$#1$-hor}}
\newcommand\bas{\text{\rm bas}}
\newcommand\bbas[1]{\text{\rm$#1$-bas}}
\newcommand\CE{\text{\rm CE}}

\begin{document}


\title[Foliated Thom isomorphism]{A Thom isomorphism in foliated de
  Rham theory}

\author{Yi Lin}

\address{Department of Mathematical Sciences, Georgia Southern
  University, Statesboro, GA 30460 USA}

\email{yilin@georgiasouthern.edu}

\author{Reyer Sjamaar}

\address{Department of Mathematics, Cornell University, Ithaca, NY
14853-4201, USA}

\email{sjamaar@math.cornell.edu}


\date\today

\dedicatory{In memory of Hans Duistermaat}


\begin{abstract}
We prove a Thom isomorphism theorem for differential forms in the
setting of transverse Lie algebra actions on foliated manifolds and
foliated vector bundles.
\end{abstract}


\maketitle

\tableofcontents


\section{Introduction}\label{section;introduction}

Any multiplicative cohomology theory has a Thom isomorphism theorem,
which relates the cohomology of the Thom space of an oriented vector
bundle to the cohomology of its base.  The Thom isomorphism theorem
for de Rham cohomology theory is well known and can be found for
instance in the textbooks~\cite[\S\,6]{bott-tu;differential-forms}
and~\cite[Ch.~I.IX]%
{greub-halperin-vanstone;connections-curvature-cohomology}.  We
present an alternative proof of the Thom isomorphism for de Rham
cohomology, which as far as we know was first given by Paradan and
Vergne in their unpublished
preprint~\cite[\S\,4]{paradan-vergne;thom-chern}.  Their proof, which
is an adaptation of Atiyah's elegant proof of Bott
periodicity~\cite{atiyah;bott-periodicity-index}, offers some
advantages: it is quite short, and it leads to an explicit homotopy
equivalence between de Rham complexes, which is useful in extending
the Thom isomorphism theorem in new directions.

As an application we establish a Thom isomorphism theorem for the
``equivariant basic'' differential forms on a foliated vector bundle
(Theorem~\ref{theorem;equivariant-basic-thom}).  These differential
forms are basic with respect to foliations on the base manifold and
the total space of a vector bundle, and they are equivariant with
respect to a ``transverse'' action of a Lie algebra on these
foliations.  This theorem is motivated by results of
T\"oben~\cite{toeben;localization-basic-characteristic} and Goertsches
et al.~\cite{goertsches-nozawa-toeben;chern-simons-foliations}, and
finds an application in our paper~\cite{lin-sjamaar;riemannian} on
cohomological localization formulas for transverse Lie algebra actions
on Riemannian foliations.  (A Thom isomorphism is also available for
the equivariant version of the Crainic-Moerdijk \v{C}ech-de Rham
complex~\cite[\S\,7]{moerdijk;groupoids-gerbes-cohomology},
\cite[\S\,2]{crainic-moerdijk;cech-de-rham-foliations}, but we have
not pursued this here.)

Constructing a Thom form in this setting turns out to be not always
possible, and necessitates a rather long excursion into the
Cartan-Chern-Weil theory of equivariant characteristic forms, which we
have placed in Appendix~\ref{section;cartan-chern-weil}.  In the
geometric cases of interest to us, where the vector bundle is the
normal bundle of a submanifold of a manifold equipped with a
Riemannian foliation, an equivariant basic Thom form exists, as we
show in \S\,\ref{section;foliated-thom-gysin}.  A consequence of this
fact is the existence of ``wrong-way'' homomorphisms and a long exact
Thom-Gysin sequence in equivariant basic de Rham theory, as we also
discuss in~\S\,\ref{section;foliated-thom-gysin}.

The authors are grateful to Xiaojun Chen and Bin Zhang for their
hospitality at the School of Mathematics of Sichuan University and to
the referee for helpful comments.

\section{Preliminaries}\label{section;preliminary}

Throughout this paper $M$ denotes a paracompact smooth manifold and
$\pi\colon E\to M$ a real vector bundle of rank $r$ with zero section
$\zeta\colon M\to E$.  For simplicity we assume the vector bundle $E$
to be oriented, noting only that the non-orientable case of the Thom
isomorphism can be handled by using forms with coefficients in the
orientation bundle as in~\cite[\S\,7]{bott-tu;differential-forms}.  A
subset $A$ of $E$ is \emph{vertically compact} if the restriction
$\pi|_A\colon A\to M$ is a proper map.  We denote the family of all
vertically compact subsets by ``$\cv$'' and say that a differential
form on $E$ is \emph{vertically compactly supported} if its support is
a member of $\cv$.

We denote the translation functor on cochain complexes by $[k]$.  So
if $(C,d)$ is a cochain complex we have $C[k]^i=C^{i+k}$ and
$d[k]=(-1)^kd$.  The \emph{algebraic mapping cone} of a morphism of
complexes $\phi\colon C\to D$ is the complex $C(\phi)$ with
$C(\phi)^i=C^i\oplus D^{i-1}$ and differential
$d(x,y)=(dx,\phi(x)-dy)$.  We have an obvious short exact sequence
$D[-1]\inj C(\phi)\sur C$; the connecting homomorphism in the
associated long exact sequence in cohomology is the map $H^i(C)\to
H^i(D)$ induced by $\phi$.  We say $\phi$ is a
\emph{quasi-isomorphism} if it induces an isomorphism in cohomology;
in other words if its mapping cone is acyclic.  A smooth map $f\colon
X\to Y$ induces a map of de Rham complexes
$f^*\colon\Omega(Y)\to\Omega(X)$, the mapping cone of which we denote
by $\Omega(f)$.  If $f$ is the inclusion of a smooth submanifold, we
call $\Omega(f)$ the \emph{relative de Rham complex} and denote it by
$\Omega(Y,X)$.

\section{The Thom isomorphism}\label{section;thom}

In this expository section we review the proof of the Thom isomorphism
theorem for differential forms on vector bundles over manifolds given
by Paradan and Vergne~\cite[\S\,4]{paradan-vergne;thom-chern}, which
is a translation into de Rham theory of Atiyah's proof of Bott
periodicity~\cite{atiyah;bott-periodicity-index}.  The trick is to
take the direct sum of two copies of the vector bundle and to exploit
the additional symmetries (the interchange map and the rotation)
produced in this way.  Given the existence of a Thom form, the Thom
isomorphism (Theorem~\ref{theorem;thom}) then follows quickly from
standard properties of differential forms.  The proof yields an
explicit, though rather long formula for a homotopy equivalence
between the two de Rham complexes, which will be of use in the next
section.  Other uses of the Paradan-Vergne idea can be found
in~\cite{fujisawa;thom-cech-de-rham}.

It was Chern~\cite{chern;gauss-bonnet} who first noticed that the Thom
form has a natural primitive defined on the complement of the zero
section, which is now called a \emph{transgression form}; see also
Mathai and Quillen~\cite[\S\,7]{mathai-quillen;thom-equivariant}.  The
Thom isomorphism can be expressed either in terms of vertically
compactly supported forms on the vector bundle $E$, or in terms of
relative differential forms on $E$ modulo the complement of the zero
section.  It is well known, but appears not to be recorded in the
literature, that these two pictures are isomorphic.  We give a proof
of this fact (Proposition~\ref{proposition;transgression}) that is
based on the notion of transgression.

\subsection{The Thom isomorphism}

The vertically compactly supported differential forms on $E$
constitute a subcomplex $\Omega_\cv(E)$ of the de Rham complex
$\Omega(E)$, and we write its cohomology as $H_\cv(E)$.  By assumption
$E$ is orientable, so every $k$-form $\beta$ in $\Omega_\cv(E)$ can be
integrated over the fibres of $E$ (see Appendix~\ref{section;fibre})
and the result is a $k-r$-form on $M$ denoted by $\pi_*\beta$.  The
projection formula~\eqref{equation;projection} shows that the fibre
integration map $\pi_*\colon\Omega_\cv(E)[r]\to\Omega(M)$ is a
morphism of graded left $\Omega(M)$-modules.  The map in cohomology
induced by $\pi_*$ is also denoted by $\pi_*$.  A \emph{Thom form} of
$E$ is an $r$-form $\tau\in\Omega_\cv^r(E)$ which satisfies
$\pi_*\tau=1$ and $d\tau=0$.  It is well known that Thom forms exist.
An explicit construction is given
in~\cite{mathai-quillen;thom-equivariant}; see also
\S\,\ref{section;foliated-thom} below.

\begin{theorem}\label{theorem;thom}
Let $E$ be an oriented vector bundle of rank $r$ over $M$.
\begin{enumerate}
\item\label{item;homotopy-equivalence}
Fibre integration $\pi_*\colon\Omega_\cv(E)[r]\longto\Omega(M)$ is a
homotopy equivalence.  A homotopy inverse of $\pi_*$ is the Thom map
\[
\zeta_*\colon\Omega(M)\longto\Omega_\cv(E)[r]
\]
defined by $\zeta_*(\alpha)=\tau\wedge\pi^*\alpha$, where
$\tau\in\Omega_\cv^r(E)$ is a Thom form of $E$.  A homotopy
$\zeta_*\pi_*\simeq\id$ is given by~\eqref{equation;homotopy-formula}
below.
\item\label{item;thom-class}
All Thom forms of $E$ are cohomologous.  Their cohomology class
$\Thom(E)\in H_\cv^r(E)$ is uniquely determined by the property that
$\pi_*(\Thom(E))=1$.
\item\label{item;free-module}
$H_\cv(E)$ is a free $H(M)$-module of rank $1$ generated by the Thom
  class $\Thom(E)$.
\end{enumerate}
\end{theorem}

\begin{proof}
\eqref{item;homotopy-equivalence}~Let $\tau$ be a Thom form of $E$.
Then $\pi_*\tau=1\in\Omega^0(M)$, so by the projection formula
\[
\pi_*\zeta_*(\alpha)=\pi_*(\tau\wedge\pi^*\alpha)=
\pi_*\tau\wedge\alpha=\alpha
\]
for all $\alpha\in\Omega(M)$.  This shows that $\pi_*\zeta_*=\id$.
Our next task is to find a cochain homotopy $\kappa$ of the complex
$\Omega_\cv(E)$ with the property
\begin{equation}\label{equation;homotopy}
\zeta_*\pi_*-{\id}=d\kappa+\kappa d.
\end{equation}
Let $\beta\in\Omega_\cv(E)$.  Then
$\zeta_*\pi_*(\beta)=\tau\wedge\pi^*\pi_*(\beta)$.  To rewrite this
expression we introduce the direct sum bundle $E\oplus E$, which has
two projection maps $p_1$, $p_2\colon E\oplus E\rightrightarrows E$
that make up a pullback diagram
\[
\begin{tikzcd}[row sep=large]
E\oplus E\ar[r,"p_2"]\ar[d,"p_1"']&E\ar[d,"\pi"]
\\
E\ar[r,"\pi"]&M.
\end{tikzcd}
\]
The pullback property (Lemma~\ref{lemma;fibre}\eqref{item;pullback})
and the projection formula give
\[
\zeta_*\pi_*(\beta)=\tau\wedge\pi^*\pi_*(\beta)=\tau\wedge
p_{1,*}p_2^*(\beta)=(-1)^rp_{1,*}(p_1^*\tau\wedge p_2^*\beta).
\]
Write elements of $E\oplus E$ as pairs $(a,b)$ and let
$\phi(a,b)=(b,a)$ be the automorphism that switches the two copies of
$E$.  Then $p_1\phi=p_2$ and $p_2\phi=p_1$, so
\begin{equation}\label{equation;fibre-thom}
\zeta_*\pi_*(\beta)=(-1)^rp_{1,*}(p_1^*\tau\wedge p_2^*\beta)=
(-1)^rp_{1,*}\phi^*(p_2^*\tau\wedge p_1^*\beta).
\end{equation}
The form $p_1^*\beta\in\Omega(E\oplus E)$ is not vertically compactly
supported with respect to the bundle projection $E\oplus E\to M$, but
the product $p_2^*\tau\wedge p_1^*\beta$ is, and therefore
\[
p_2^*\tau\wedge p_1^*\beta\in\Omega_\cv(E\oplus E).
\]
The automorphism $\phi$ of $E\oplus E$ is the composition
$\phi=\varrho_1\circ\sigma$ of the quarter-turn map
$\varrho_1(a,b)=(-b,a)$ and the reflection $\sigma(a,b)=(a,-b)$.  The
map $\varrho_1$ is homotopic to the identity through the family of
rotations $\varrho\colon[0,1]\times(E\oplus E)\to E\oplus E$ defined
by
\[
\textstyle\varrho(t,a,b)=\bigl(\cos\bigl(\frac12\pi
t\bigl)a-\sin\bigl(\frac12\pi t\bigr)b,\sin\bigl(\frac12\pi
t\bigr)a+\cos\bigl(\frac12\pi t\bigr)b\bigr).
\]
By Corollary~\ref{corollary;homotopy} this homotopy induces a cochain
homotopy $p_*\varrho^*$ of the complex $\Omega(E\oplus E)$,
\begin{equation}\label{equation;rotation}
\varrho_1^*-{\id}=dp_*\varrho^*+p_*\varrho^*d,
\end{equation}
where $p\colon[0,1]\times(E\oplus E)\to E\oplus E$ is the projection.
The homotopy $p_*\varrho^*$ preserves forms that are vertically
compactly supported with respect to the bundle projection $E\oplus
E\to M$, so~\eqref{equation;rotation} is valid as a homotopy of the
complex $\Omega_\cv(E\oplus E)$.  Since $\phi=\varrho_1\circ\sigma$,
this yields
\[\phi^*=\sigma^*({\id}+dp_*\varrho^*+p_*\varrho^*d).\]
We substitute this formula into~\eqref{equation;fibre-thom},
\begin{align*}
\zeta_*\pi_*(\beta)&=
(-1)^rp_{1,*}\sigma^*({\id}+dp_*\varrho^*+p_*\varrho^*d)(p_2^*\tau\wedge
p_1^*\beta)\\
&=p_{1,*}({\id}+dp_*\varrho^*+p_*\varrho^*d)(p_2^*\tau\wedge
p_1^*\beta)\\
&=\pi^*\pi_*(\tau)\wedge\beta+(p_{1,*}dp_*\varrho^*+p_{1,*}p_*\varrho^*d)
(p_2^*\tau\wedge p_1^*\beta)\\
&=\beta+(-1)^r\bigl(dp_{1,*}p_*\varrho^*l(p_2^*\tau)p_1^*+
p_{1,*}p_*\varrho^*l(p_2^*\tau)p_1^*d\bigr)(\beta)\\
&=\beta+(d\kappa+\kappa d)(\beta),
\end{align*}
and so we see that~\eqref{equation;homotopy} holds.  Here we used the
fact that $p_{1,*}\sigma^*=(-1)^rp_{1,*}$ (since $\sigma$ reverses the
fibres of $p_1$), the projection formula, the fact that
$[p_{1,*},d]=0$ (Lemma~\ref{lemma;fibre}\eqref{item;stokes}), and the
fact that $d\tau=0$.  The notation $l(\gamma)$ means left
multiplication by a form $\gamma$, and the formula for the homotopy is
\begin{equation}\label{equation;homotopy-formula}
\kappa=(-1)^rp_{1,*}\circ p_*\circ\varrho^*\circ l(p_2^*\tau)\circ
p_1^*.
\end{equation}

\eqref{item;thom-class}~This assertion follows immediately from
\eqref{item;homotopy-equivalence}.

\eqref{item;free-module}~It follows from
\eqref{item;homotopy-equivalence} that the Thom map in cohomology
$\zeta_*\colon H(M)\to H_\cv(E)[r]$ is an isomorphism of
$H^*(M)$-modules.  Hence $H_\cv(E)[r]$ is free on the single generator
$\zeta_*(1)=[\tau]$.
\end{proof}  

\subsection{Transgression and relative forms}

The \emph{punctured vector bundle} is $E^\times=E\backslash\zeta(M)$.
Define a projection $p$ and a homotopy $h$
\[
\begin{tikzcd}
E^\times&\ar[l,"p"'][1,\infty)\times E^\times\ar[r,"h"]&E
\end{tikzcd}
\]
by $p(t,v)=t$ and $h(t,v)=tv$.  For every vertically compact subset
$A$ of $E$ the map $p\colon h^{-1}(A)\to E^\times$ is proper; in other
words the subset $h^{-1}(A)$ of $[1,\infty)\times E^\times$ is
  vertically compact for $p$.  Therefore we have a well-defined map
\[\phi=p_*h^*\colon\Omega_\cv(E)\longto\Omega(E^\times)[-1],\]
which we call the \emph{transgression map}.
Lemma~\ref{lemma;fibre}~\eqref{item;stokes} shows that
$[d,\phi]=h_1^*$.  Since $h_1$ is the inclusion $E^\times\to E$, we
obtain the \emph{transgression formula}
\begin{equation}\label{equation;transgression}
\beta|_{E^\times}=d\phi(\beta)+\phi(d\beta)
\end{equation}
for all $\beta\in\Omega_\cv(E)$.  In particular, if $\beta$ is closed
(for instance the Thom form), its restriction to $E^\times$ has a
primitive $\phi(\beta)$.  (If we regard the map $h$ as a homotopy
$[0,\infty)\times E^\times\to E$, as
  in~\cite[\S\,7]{mathai-quillen;thom-equivariant}, we obtain a
  similar formula for $h_0^*\beta=(\pi^*\zeta^*\beta)|_{E^\times}$.)

The following proposition is a folklore result, which provides several
alternative models for vertically compactly supported cohomology.
Recall (see~\S\,\ref{section;preliminary}) that the relative de Rham
complex $\Omega(E,E^\times)$ is the mapping cone of the restriction
map $\Omega(E)\to\Omega(E^\times)$.  The transgression
formula~\eqref{equation;transgression} can be interpreted as saying
that the map
\[\psi\colon\Omega_\cv(E)\longto\Omega(E,E^\times)\]
defined by $\psi(\beta)=(\beta,\phi(\beta))$ is a morphism of
complexes.  The result asserts that $\psi$ induces an isomorphism in
cohomology.  There are two other complexes quasi-isomorphic to
$\Omega_\cv(E)$, the definition of which involves the choice of a
Riemannian fibre metric on $E$.  Let $BE$ be the unit disc bundle and
$SE=\partial(BE)$ the unit sphere bundle with respect to this metric.
In order not to overload the notation, we will use $j$ for any of the
four inclusion maps
\[
BE\longinj E,\quad SE\longinj
E^\times,\quad(BE,SE)\longinj(E,E^\times),\quad
SE\longinj\overline{E\backslash BE},
\]
where $\overline{E\backslash BE}$ denotes the closure of $E\backslash
BE$ in $E$.  Then we have a restriction map
$j^*\colon\Omega(E,E^\times)\to\Omega(BE,SE)$.  Let $\Omega_{BE}(E)$
be the collection of all differential forms on $E$ which have support
contained in $BE$.  Then $\Omega_{BE}(E)$ is a subcomplex of
$\Omega_\cv(E)$ and we denote the inclusion map by
$j_*\colon\Omega_{BE}(E)\inj\Omega_\cv(E)$.

\begin{proposition}\label{proposition;transgression}
The morphisms
\[
\begin{tikzcd}
\Omega_{BE}(E)\ar[r,"j_*"]& \Omega_\cv(E)\ar[r,"\psi"]&
\Omega(E,E^\times)\ar[r,"j^*"]&\Omega(BE,SE)
\end{tikzcd}
\]
are quasi-isomorphisms.
\end{proposition}

\begin{proof}
There exists a Thom form $\tau$ which is supported on $BE$, as shown
for instance in \S\,\ref{section;foliated-thom} below.  With such a
choice of $\tau$ we have
$\zeta_*(\alpha)=\tau\wedge\pi^*\alpha\in\Omega_{BE}(E)$ for all
$\alpha\in\Omega(M)$.  Let us write $\zeta_*^B$ for $\zeta_*$ viewed
as a map from $\Omega(M)$ into $\Omega_{BE}(E)$.  Then
$\zeta_*=j_*\circ\zeta_*^B$.  Theorem~\ref{theorem;thom} says that
$\zeta_*$ is a quasi-isomorphism, so to prove the proposition it
suffices to show that the maps
\begin{align}
\label{item;thom-ball}
\zeta_*^B\colon\Omega(M)[-r]&\longto\Omega_{BE}(E),\\
\label{item;ball-sphere}
j^*\circ\psi\circ j_*\colon\Omega_{BE}(E)&\longto\Omega(BE,SE),\\
\label{item;retract}
j^*\colon\Omega(E,E^\times)&\longto\Omega(BE,SE)
\end{align}
are quasi-isomorphisms.  The map $\zeta_*^B$ has a left inverse
$\pi_*^B$, defined by integration over the fibres of the projection
$\pi^B\colon BE\to M$.  We also have the formula
\begin{equation}\label{equation;ball-homotopy-formula}
\zeta_*^B\pi_*^B(\beta)-\beta=(d\kappa+\kappa d)(\beta)
\end{equation}
for all $\beta\in\Omega_{BE}(E)$, where $\kappa$ is as
in~\eqref{equation;homotopy-formula}.  But the support of
$\kappa(\beta)\in\Omega(E)$ may not be contained in $BE$,
so~\eqref{equation;ball-homotopy-formula} is not valid as a homotopy
of the complex $\Omega_{BE}(E)$.  Indeed, the support of
$p_2^*\tau\wedge p_1^*\beta$ is contained in the ``bi-disc'' bundle
$BE\times_EBE$, so we see from~\eqref{equation;homotopy-formula} that
the support of $\kappa(\beta)$ is a subset of
$p_1p\varrho^{-1}(BE\times_EBE)$, which is contained in the disc
bundle $B_RE$ of radius $R=\sqrt2$.  (Here the maps $p_1$, $p_2$,
$\rho$, and $p$ are as in the proof of Theorem~\ref{theorem;thom}.)
We fix this by retracting the disc bundle $B_RE$ into the disc bundle
$BE$: we have a homotopy $\bar{h}$ and a projection $\bar{p}$
\[
\begin{tikzcd}
E&\ar[l,"\bar{p}"'][1,R]\times E\ar[r,"\bar{h}"]&E
\end{tikzcd}
\]
defined by $h(t,v)=tv$ and $p(t,v)=t$.  By
Lemma~\ref{lemma;fibre}~\eqref{item;stokes} this homotopy gives rise
to a homotopy formula on the de Rham complex $\Omega(E)$, namely
$h_R^*-\id_E^*=[\lambda,d]$, where $\lambda=\bar{p}_*\circ\bar{h}^*$.
Combining this with~\eqref{equation;ball-homotopy-formula} gives
\begin{equation}\label{equation;shrunk-ball-homotopy-formula}
\zeta_*^B\pi_*^B(\beta)-\beta=(d\mu+\mu d)(\beta)
\end{equation}
for all $\beta\in\Omega_{BE}(E)$, where
$\mu=h_R^*\circ\kappa+\lambda\circ(\zeta_*^B\pi_*^B-id_E^*)$.  From
$\supp(\beta)\subseteq BE$ we get $\supp(h_R^*\kappa(\beta))\subseteq
h_R^{-1}(B_RE)=BE$ and
$\supp(\lambda(\beta))\subseteq\bar{p}\bar{h}^{-1}(BE)=BE$, so
$\supp(\mu(\beta))\subseteq BE$.
So~\eqref{equation;shrunk-ball-homotopy-formula} shows that $\pi^B_*$
is a homotopy inverse of $\zeta^B_*$, and hence~\eqref{item;thom-ball}
is a quasi-isomorphism.  Next we show~\eqref{item;ball-sphere} is a
quasi-isomorphism.  Let $\beta\in\Omega_\cv(E)$ and let $A$ be the
support of $\beta$.  Then the support of $\phi(\beta)=p_*h^*(\beta)$
is contained in $p(h^{-1}(A))=\bigcup_{t\ge1}\frac1tA$.  It follows
that the transgression $\phi$ maps $\Omega_{BE}(E)$ to
$\Omega_{BE}(E)$.  In particular, if $\beta\in\Omega_{BE}(E)$, then
$\phi(\beta)=0$ on $SE$ and hence $j^*\psi
j_*(\beta)=j^*(\beta,\phi(\beta))=(\beta,0)$.  In other words the
map~\eqref{item;ball-sphere} is just the map induced by the inclusion
$j\colon BE\to E$.  The complexes $\Omega_{BE}(E)$ and $\Omega(BE,SE)$
fit into two short exact sequences
\begin{gather*}
\begin{tikzcd}[ampersand replacement=\&]
\Omega_{BE}(E)\ar[r,hook]\&\Omega(E)\ar[r,two heads]\&
\Omega\bigl(\overline{E\backslash BE}\bigr),
\end{tikzcd}
\\
\begin{tikzcd}[ampersand replacement=\&]
\Omega(SE)[-1]\ar[r,hook]\&\Omega(BE,SE)\ar[r,two heads]\&\Omega(BE).
\end{tikzcd}
\end{gather*}
We have two corresponding long exact cohomology sequences and a map
between them,
\[
\begin{tikzcd}[column sep=small]
\cdots\ar[r]&H^k_{BE}(E)\ar[r]\ar[d,"j^*"']&H^k(E)\ar[r]\ar[d,"j^*"']&
H^k\bigl(\overline{E\backslash
  BE}\bigr)\ar[r]\ar[d,"j^*"']&H^{k+1}_{BE}(E)\ar[r]\ar[d,"j^*"']&\cdots
\\
\cdots\ar[r]&H^k(BE,SE)\ar[r]&H^k(BE)\ar[r]&H^k(SE)\ar[r]&
H^{k+1}(BE,SE)\ar[r]&\cdots
\end{tikzcd}
\]
The inclusions $BE\to E$ and $SE\to\overline{E\backslash BE}$ are
homotopy equivalences, so by the homotopy lemma,
Corollary~\ref{corollary;homotopy}, they induce isomorphisms
$H^k(E)\cong H^k(BE)$ and $H^k\bigl(\overline{E\backslash
  BE}\bigr)\cong H^k(SE)$.  Hence $j^*\colon H^k_{BE}(E)\to
H^k(BE,SE)$ is also an isomorphism, i.e.~\eqref{item;ball-sphere} is a
quasi-isomorphism.  The proof that~\eqref{item;retract} is a
quasi-isomorphism is similar but easier, and uses the fact that the
inclusions $BE\to E$ and $SE\to E^\times$ are homotopy equivalences.
\end{proof}

\subsection{The Thom-Gysin sequences}

The Thom isomorphism theorem leads to two exact sequences,
Propositions~\ref{proposition;thom-gysin-vector}
and~\ref{proposition;thom-gysin-submanifold} below, each of which is
known as the \emph{Thom-Gysin sequence}.  If $\tau\in\Omega_\cv^r(E)$
is a Thom form of the bundle $E$, then
$\eta=\zeta^*\tau=\zeta^*\zeta_*1\in\Omega^r(M)$ is called the
\emph{Euler form} corresponding to $\tau$.  Its class
$\Eu(E)=[\eta]\in H^r(M)$ is the \emph{Euler class} of $E$.  Let $SE$
be the sphere bundle of $E$ with respect to some Riemannian fibre
metric on $E$ and let $\pi_{SE}=\pi|_{SE}\colon SE\to M$ be the
projection.

\begin{proposition}\label{proposition;thom-gysin-vector}
The sequence
\[
\begin{tikzcd}
\cdots\ar[r]&H^{k-r}(M)\ar[r,"\Eu(E)\cup"]&[2em]
H^k(M)\ar[r,"\pi_{SE}^*"]&H^k(SE)\ar[r,"\pi_{SE,*}"]&
H^{k-r+1}(M)\ar[r]&\cdots
\end{tikzcd}
\]
is exact.
\end{proposition}

\begin{proof}
Make the following substitutions into the long exact sequence for the
pair $(BE,SE)$:
\[
\begin{tikzcd}
\cdots\ar[r]&H^k(BE,SE)\ar[r]&H^k(BE)\ar[r,"i_{SE}^*"]&
H^k(SE)\ar[r,"\delta"]&H^{k+1}(BE,SE)\ar[r]&\cdots\\
&H^{k-r}(M)\ar[u,"\zeta_*","\cong"']&H^k(M)\ar[u,leftarrow,shift
  left=1.5ex,"\zeta^*","\cong"']\ar[u,shift right=1ex,"\pi_{BE}^*"']&&
H^{k-r+1}(M)\ar[u,leftarrow,"\pi_{BE,*}","\cong"']
\end{tikzcd}
\]
Here $\pi_{BE}\colon BE\to M$ is the projection and $i_{SE}\colon
SE\to BE$ is the inclusion.  The isomorphism $H^k(BE,SE)\cong
H^{k-r}(M)$ follows from Theorem~\ref{theorem;thom} and
Proposition~\ref{proposition;transgression} and the isomorphism
$H^k(BE)\cong H^k(M)$ follows from the fact that $\pi_{BE}\colon BE\to
M$ is a deformation retraction.  The map $H^{k-r}(M)\to H^k(M)$ is
given by
\[
\alpha\longmapsto\zeta^*\zeta_*\alpha=
\zeta^*(\tau\wedge\pi_{BE}^*\alpha)=\eta\wedge\alpha,
\]
and the map $H^k(M)\to H^k(SE)$ is
$i_{SE}^*\circ\pi_{BE}^*=\pi_{SE}^*$.  The map
$\pi_{BE,*}\circ\delta\colon H^k(SE)\to H^{k-r+1}(M)$ can be described
as follows.  Let $\lambda$ be a closed $k$-form on $SE$, extend it to
a form $\mu$ on $BE$; then $\nu=(d\mu,\lambda)$ is a cocycle in
$\Omega^{k+1}(BE,SE)$ and $\delta([\alpha])=[\nu]$.  Also
\[
\pi_{BE,*}\nu=\pi_{BE,*}d\mu=\pi_{SE,*}\lambda+(-1)^rd\pi_{BE,*}\mu,
\]
where the second equality follows from
Lemma~\ref{lemma;fibre}\eqref{item;stokes}.  We conclude that
$\pi_{BE,*}\delta[\lambda]=\pi_{SE,*}[\lambda]$.
\end{proof}

The de Rham complex has the following excision property.

\begin{lemma}\label{lemma;excision}
Let $X$ be a manifold and $i\colon Y\to X$ a closed submanifold with
normal bundle $N=i^*TX/TY$.  Choose a tubular neighbourhood embedding
$i_N\colon N\to X$.  Then $i_N^*\colon\Omega(X,X\backslash
Y)\to\Omega(N,N^\times)$ is a quasi-isomorphism.
\end{lemma}

\begin{proof}
Put $Z=X\backslash Y$ and let $\{U,V\}$ be an open cover of $X$.  We
have two short exact Mayer-Vietoris sequences of complexes
\begin{gather*}
\Omega(X)\longinj\Omega(U)\oplus\Omega(V)\longsur\Omega(U\cap V),\\
\Omega(Z)\longinj\Omega(Z\cap U)\oplus\Omega(Z\cap
V)\longsur\Omega(Z\cap U\cap V).
\end{gather*}
Forming mapping cones we obtain an exact sequence of relative de Rham
complexes
\[
\Omega(X,Z)\longinj\Omega(U,Z\cap U)\oplus\Omega(V,Z\cap
V)\longsur\Omega(U\cap V,Z\cap U\cap V).
\]
Taking $U=i_N(N)$ and $V=Z$ yields the exact sequence
\[
\Omega(X,Z)\longinj\Omega(N,N^\times)\oplus\Omega(Z,Z)\longsur
\Omega(N^\times,N^\times).
\]
Writing the corresponding long exact cohomology sequence gives the
result.
\end{proof}

In the setting of Lemma~\ref{lemma;excision} suppose that $Y$ is
co-oriented, let $r$ be the codimension of $Y$, and choose a Thom form
$\tau(N)\in\Omega^r(N)$ of $N$.  For each $\beta\in\Omega_\cv(N)$ the
form $(i_N^{-1})^*\beta$ is supported near $Y$, so extends by zero to
a form on $X$, which we denote by $i_{N,*}\beta$.  The \emph{wrong-way
  homomorphism} is the degree $r$ morphism of complexes
\begin{equation}\label{equation;wrong-way}
i_*=i_{N,*}\circ\zeta_{N,*}\colon\Omega(Y)[-r]\longto\Omega(X)
\end{equation}
defined by $i_*\alpha=i_{N,*}(\tau(N)\wedge\pi_N^*\alpha)$.

\begin{proposition}\label{proposition;thom-gysin-submanifold}
Let $X$ be a manifold, let $Y$ be a closed co-oriented submanifold of
codimension $r$, and let $i\colon Y\to X$ and $j\colon X\backslash
Y\to X$ be the inclusions.  We have a long exact sequence
\[
\begin{tikzcd}
\cdots\ar[r]&H^{k-r}(Y)\ar[r,"i_*"]&H^k(X)\ar[r,"j^*"]&
H^k(X\backslash Y)\ar[r]&H^{k-r+1}(Y)\ar[r]&\cdots
\end{tikzcd}
\]
\end{proposition}

\begin{proof}
In the long exact sequence for the pair $(X,X\backslash Y)$ the term
$H^k(X,X\backslash Y)$ is isomorphic to $H^k(N,N^\times)$ by
Lemma~\ref{lemma;excision}, which is isomorphic to $H^{k-r}(Y)$ by
Theorem~\ref{theorem;thom} and
Proposition~\ref{proposition;transgression}.
\end{proof}

\section{The equivariant basic Thom isomorphism}
\label{section;foliated-thom}

As in the previous section, $\pi\colon E\to M$ denotes a smooth
oriented real vector bundle over a manifold $M$.  The treatment of the
Thom isomorphism theorem given in \S\,\ref{section;thom} can be
readily modified to apply to various subcomplexes and extensions of
the de Rham complex.  For instance, if the vector bundle is
equivariant with respect to a compact Lie group $G$, it leads to a new
proof of the Thom isomorphism in $G$-equivariant de Rham theory.  This
proof has the advantage over proofs such as the one given
in~\cite[\S\,10.6]{guillemin-sternberg;supersymmetry-equivariant} that
it offers an explicit homotopy equivalence between the $G$-equivariant
de Rham complexes $\Omega_{G,\cv}(E)[r]$ and $\Omega_G(M)$.

In this section, for the purpose of our
work~\cite{lin-sjamaar;riemannian}, we will consider a generalization
of this $G$-equivariant case, namely a situation where both $M$ and
$E$ are equipped with foliations and where a finite-dimensional Lie
algebra $\g$ acts transversely on the foliated manifolds $M$ and $E$.
We obtain an ``equivariant basic'' Thom isomorphism,
Theorem~\ref{theorem;equivariant-basic-thom}, for differential forms
which are equivariant with respect to the action of $\g$ and basic
with respect to the foliations.  This theorem can be regarded as a
substitute for a Thom isomorphism for $\g$-equivariant vector bundles
over the leaf space of $M$, which is usually not a manifold but a
topological space with poor properties.  It extends a result of
T\"oben~\cite[\S\,6]{toeben;localization-basic-characteristic}, who
considered the case where the foliation is Killing and the Lie algebra
$\g$ is Molino's structural Lie algebra of the foliation.

A Thom class does not necessarily exist in this context.  In
\S\,\ref{subsection;without} we offer a simple example where it does
not exist and in \S\,\ref{subsection;existence} we state a sufficient
condition for when it does exist.  For our purposes the most important
case where an equivariant basic Thom class exists is when $M$ is a
submanifold of a Riemannian foliated manifold and $E$ is the normal
bundle of $M$, as we will discuss in
\S\,\ref{section;foliated-thom-gysin}.

We start with a review of some topics in foliation theory
(\S\S\,\ref{subsection;foliation}--\ref{subsection;characteristic}),
standard references for which
include~\cite{kamber-tondeur;foliated-bundles},
\cite{moerdijk-mrcun;foliations-groupoids},
\cite{molino;riemannian-foliations},
and~\cite{tondeur;geometry-foliations}.

\subsection{Foliations}\label{subsection;foliation}

By a \emph{foliation} of $M$ we mean a smooth and regular
(i.e.\ constant rank) foliation.  We denote the $0$-dimensional
foliation of $M$ by $*$ or~$*_M$.  The tangent bundle of a foliation
$\F$ is written as $T\F$, and the leaf of a point $x\in M$ as $\F(x)$.
We call the dimension of the leaves $\F(x)$ the \emph{rank} or
\emph{dimension} of the foliation.  If $\F$ is a foliation of $M$ and
$U$ an open subset of $M$, the \emph{restricted foliation} $\F|_U$ is
the foliation of $U$ whose leaves are the connected components of the
intersections $U\cap\F(x)$.  We have $T(\F|_U)=(T\F)|_U$.  We say that
a smooth map $f\colon M\to M'$ of foliated manifolds $(M,\F)$ and
$(M',\F')$ is \emph{foliate} if $f$ maps each leaf $\F(x)$ of $\F$ to
the leaf $\F'(f(x))$ of $\F'$, and that $f$ is a \emph{foliate
  isomorphism} if $f$ is a diffeomorphism and $f$ and $f^{-1}$ are
foliate.

The space of sections of the tangent bundle of a foliation $\F$ is a
Lie subalgebra $\X(\F)$ of the Lie algebra of vector fields $\X(M)$.
Let $\lie{N}(\F)=N_{\X(M)}(\X(\F))$ be the normalizer of this
subalgebra.  Elements of $\lie{N}(\F)$ are \emph{foliate vector
  fields}, i.e.\ vector fields whose flow consists of foliate maps.
The quotient $\lie{N}(\F)/\X(\F)$ is a Lie algebra, which we denote by
$\X(M,\F)$, and the elements of which we call \emph{transverse vector
  fields}.  Thus a transverse vector field is not a vector field, but
an equivalence class of foliate vector fields modulo $\X(\F)$.  The
flow of a transverse vector field is well-defined only up to a flow
along the leaves of~$\F$.

The notion of a transverse vector field is a substitute for the
ill-defined notion of a vector field on the leaf space $M/\F$.  We say
that the leaf space \emph{is a manifold} if $M/\F$ is equipped with a
smooth structure with respect to which the quotient map $M\to M/\F$ is
a submersion.  If $M/\F$ is a manifold, transverse vector fields on
$M$ can be identified naturally with vector fields on~$M/\F$.

\subsection{Basic differential forms}\label{subsection;basic}

Let $\F$ be a foliation of the manifold $M$.  A differential form
$\alpha$ on $M$ is \emph{$\F$-basic} if its Lie derivatives
$L(u)\alpha$ and contractions $\iota(u)\alpha$ vanish for all vector
fields $u$ on $M$ that are tangent to $\F$.  The set of $\F$-basic
forms is a differential graded subalgebra of the de Rham complex
$\Omega(M)$, which we denote by $\Omega(M,\F)$.  Its cohomology is a
graded commutative algebra called the \emph{$\F$-basic de Rham
  cohomology} and denoted by $H(M,\F)$.  If the leaf space $M/\F$ is a
manifold, then the basic de Rham complex of $(M,\F)$ is isomorphic to
the de Rham complex of $M/\F$.  A foliate map
$f\colon(M,\F)\to(M',\F')$ induces a pullback morphism of differential
graded algebras $f^*\colon\Omega(M',\F')\to\Omega(M,\F)$ and hence a
morphism of graded algebras $f^*\colon H(M',\F')\to H(M,\F)$.

\subsection{Foliated bundles}\label{subsection;bundle}

Let $(M,\F=\F_M)$ and $(P,\F_P)$ be foliated manifolds.  Let
$\pi\colon P\to M$ be a smooth map and let $F$ be a third manifold.
By a \emph{foliated fibre bundle chart on $P$ with fibre $F$} we mean
a pair $(U,\phi)$, where $U$ is an open subset of $M$ and $\phi$ is a
diffeomorphism
$\begin{tikzcd}[cramped,sep=small]
  \phi\colon\pi^{-1}(U)\ar[r,"\cong"]&U\times F
\end{tikzcd}$
which satisfies ${\pr_1}\circ\phi=\pi$ (with $\pr_1\colon U\times F\to
U$ being the projection onto the first factor) and which is a foliate
isomorphism with respect to the restricted foliation
$\F_P|_{\pi^{-1}(U)}$ of $\pi^{-1}(U)$ and the product foliation
$\F|_U\times{*}_F$ of $U\times F$.  We say that $P$ is a
\emph{foliated fibre bundle with fibre $F$} if $P$ is equipped with a
\emph{foliated fibre bundle atlas with fibre $F$}, i.e.\ a collection
of foliated fibre bundle charts $(U_i,\phi_i)$ with fibre $F$ whose
domains $U_i$ form an open cover of~$M$.

We single out some special classes of foliated fibre bundles.  If the
fibre is a Lie group $K$ and the foliated fibre bundle charts
$(U_i,\phi_i)$ are principal $K$-bundle charts, we say $P$ is a
\emph{foliated principal $K$-bundle}.  If the fibre is a (real) vector
space and the foliated fibre bundle charts $(U_i,\phi_i)$ are vector
bundle charts, we say $P$ is a \emph{foliated vector bundle}.  If $P$
is a foliated principal $K$-bundle over $M$ and $F$ is any manifold on
which $K$ acts smoothly, then the quotient $Q=(P\times F)/K$ by the
diagonal $K$-action has a natural structure of a foliated fibre bundle
over $P/K=M$ with fibre $F$.  We call a bundle $Q$ that arises in this
way a \emph{foliated fibre bundle with structure group $K$}.  In
particular a foliated vector bundle is a foliated fibre bundle with
fibre $F=\R^r$ and structure group $K=\GL(r,\R)$.

Let $\pi\colon(P,\F_P)\to(M,\F_M)$ be a foliated fibre bundle with
structure group a finite-dimensional Lie group $K$.  The bundle
projection $\pi$ is a foliate map and the foliations $\F_M$ and $\F_P$
have the same rank.  Indeed, for each $p\in P$ the tangent map
$T_p\pi\colon T_pP\to T_{\pi(p)}M$ maps the tangent space $T_p\F_P$ to
the leaf $\F_P(p)$ isomorphically onto the tangent space
$T_{\pi(p)}\F_M$ of the leaf $\F_M(\pi(p))$.  Hence for each $p\in P$
the restriction of $\pi$ to the leaf $\F_P(p)$ is a local
diffeomorphism $\F_P(p)\to\F_M(\pi(p))$, and $\pi$ induces a vector
bundle isomorphism
$\begin{tikzcd}[cramped,sep=small] T\F_P\ar[r,"\cong"]&\pi^*T\F_M
\end{tikzcd}$.
The inverse of this isomorphism is a map
\begin{equation}\label{equation;partial-connection}
\pi^*T\F_M\longto TP,
\end{equation}
known as the \emph{partial connection} of the foliated bundle $P$,
whose image is equal to $T\F_P$.  The partial connection gives us, for
each vector field $v$ tangent to $\F_M$, a unique vector field $v_P$
tangent to $\F_P$ which is $\pi$-related to $v$.  The map $v\mapsto
v_P$ is a Lie algebra homomorphism
\[\X(\F_M)\longto\X(\F_P)\]
called the \emph{horizontal lifting homomorphism} of the partial
connection.  It follows that, for each leaf $L$ of $\F_M$, the partial
connection restricts to a genuine (Ehresmann) connection $\pi^*(TL)\to
T(P|_L)$ on $P|_L$ which is flat.  Hence for each $p\in P$ the map
$\F_P(p)\to\F_M(\pi(p))$ is a Galois covering map, whose covering
group is the holonomy of the connection.

\begin{remark}\label{remark;quotient}
Suppose that the leaf space $\bar{M}=M/\F_M$ is a manifold and that
the holonomy of the partial connection on $P$ is trivial for all
leaves of $M$.  Then, by \cite[\S\,2.6,
  Lemma~2.5]{molino;riemannian-foliations}, there exists a pair
$(\bar{P},q_P)$, unique up to isomorphism, consisting of a fibre
bundle $\bar{P}\to\bar{M}$ with structure group $K$ and an isomorphism
$\begin{tikzcd}[cramped,sep=small]q_P\colon
  P\ar[r,"\cong"]&q^*\bar{P}\end{tikzcd}$.  Every foliated fibre
bundle over $M$ with structure group $K$ is locally, in a foliation
chart on $M$, isomorphic to a pullback of this kind.
\end{remark}

An instance of a foliated principal bundle over $M$ that we shall
frequently return to is the \emph{transverse frame bundle} $P$ of the
foliation $\F$, which is defined as the frame bundle of the normal
bundle $N\F=TM/T\F$ of the foliation, and which is a principal
$K=\GL(q,\R)$-bundle, where $q$ is the codimension of the foliation
$\F$.  We explain briefly how the foliation of $P$ comes about;
see~\cite[\S\,2.4]{molino;riemannian-foliations} for details.  The
flow $\phi_t$ of a foliate vector field $v\in\lie{N}(\F)$ is a local
$1$-parameter group of foliate diffeomorphisms of $M$, and so the
tangent flow $T\phi_t$ induces a local $1$-parameter group of vector
bundle automorphisms of $N\F$, which lifts naturally to a
$K$-equivariant flow $\phi_{P,t}$ of bundle automorphisms of $P$.  The
infinitesimal generator $\pi^\dag(v)$ of $\phi_{P,t}$ is a
$K$-invariant vector field on $P$, and the map $v\mapsto\pi^\dag(v)$
is a homomorphism of Lie algebras
$\pi^\dag\colon\lie{N}(\F)\longto\X(P)^K$.  The restriction of
$\pi^\dag$ to $\X(\F)$ is a homomorphism $\X(\F)\to\X(P)$, which is
the horizontal lifting map for a partial connection $\pi^*T\F\to TP$
on $P$.  The image of the map $\pi^*T\F\to TP$ is the tangent bundle
to the foliation $\F_P$ that makes $P$ a foliated principal
$K$-bundle.  We call the Lie algebra map
\begin{equation}\label{equation;natural-lifting}
\pi^\dag\colon\X(M,\F)\longto\X(P,\F_P)^K
\end{equation}
induced by $\pi^\dag$ the \emph{natural lifting homomorphism} of the
transverse frame bundle.  The normal bundle $N\F$ is isomorphic to the
associated bundle $(P\times\R^q)/K$ and so inherits the structure of a
foliated vector bundle from $P$.

\subsection{Transverse Lie algebra actions and equivariant basic
  differential forms}\label{subsection;transverse}

Let $(M,\F)$ be a foliated manifold and let $\g$ be a
finite-dimensional real Lie algebra.  A \emph{transverse action} of
$\g$ on $(M,\F)$ is a Lie algebra homomorphism $a$ from $\g$ to the
Lie algebra of transverse vector fields $\X(M,\F)$.  If $a$ is a
transverse $\g$-action on $M$, we call the triple $(M,\F,a)$ a
\emph{foliated $\g$-manifold}.  The notion of a transverse Lie algebra
action on a foliated manifold is a substitute for the ill-defined
notion of a Lie algebra action on the leaf space.  If the leaf space
$M/\F$ is a manifold, a transverse $\g$-action on $M$ amounts to a
$\g$-action on $M/\F$.

Let $(M,\F,a)$ be a foliated $\g$-manifold.  For $\xi\in\g$ let
$\xi_M=a(\xi)\in\X(M,\F)$ denote the transverse vector field on $M$
defined by the $\g$-action.  For $\alpha\in\Omega(M,\F)$ define
\[
\iota(\xi)\alpha=\iota(\tilde{\xi}_M)\alpha,\qquad
L(\xi)\alpha=L(\tilde{\xi}_M)\alpha,
\]
where $\tilde{\xi}_M\in\lie{N}(\F)$ is a foliate vector field that
represents $\xi_M$.  Since $\alpha$ is $\F$-basic, these contractions
and derivatives are independent of the choice of the representative
$\tilde{\xi}_M$ of $\xi_M$.  Goertsches and
T\"oben~\cite[Proposition~3.12]%
{goertsches-toeben;equivariant-basic-riemannian} observed that they
obey the usual rules of \'E. Cartan's differential calculus, namely
$[L(\xi),L(\eta)]=L([\xi,\eta])$ etc.  In other words the transverse
$\g$-action makes the basic de Rham complex $\Omega(M,\F)$ a
\emph{$\g$-differential graded algebra} in the sense
of~\cite[Appendice]{sergiescu;cohomologie-basique}
or~\cite{alekseev-meinrenken;lie-chern-weil}.  (These objects are
called \emph{$\g^\star$-algebras}
in~\cite[Ch.\ 2]{guillemin-sternberg;supersymmetry-equivariant}).  See
\S\S\,\ref{section;tildeg}--\ref{section;gtilde-algebras} for a
detailed definition of $\g$-differential graded modules and algebras.
Like any other $\g$-differential graded module, the $\F$-basic de Rham
complex $\Omega(M,\F)$ has a \emph{Weil complex}
\[
\Omega_\g(M,\F)=(\W\g\otimes\Omega(M,\F))_\bbas\g.
\]
Here $\W\g$ denotes the Weil algebra of $\g$, which is a differential
graded commutative algebra isomorphic to $S\g^*\otimes\Lambda\g^*$ as
an algebra, and ``$\bbas\g$'' means ``$\g$-basic subcomplex''.  (See
\S\S\,\ref{section;weil}--\ref{section;equivariant}.)  We refer to
elements of $\Omega_\g(M,\F)$ as \emph{$\g$-equivariant $\F$-basic
  differential forms}.  The cohomology of the Weil complex is the
\emph{$\g$-equivariant $\F$-basic de Rham cohomology} $H_\g(M,\F)$ of
the foliated manifold with respect to the transverse action.  We will
frequently abbreviate the cumbersome phrase ``$\g$-equivariant
$\F$-basic'' to ``equivariant basic''.

Let $(M',\F',a')$ be another foliated $\g$-manifold.  We say that a
foliate map $f\colon M\to M'$ is \emph{$\g$-equivariant} if the
transverse vector fields $\xi_M$ and $\xi_{M'}$ are $f$-related for
all $\xi\in\g$.  A $\g$-equivari\-ant foliate map $f$ induces a
pullback map of $\g$-differential graded algebras
$\Omega(M',\F')\to\Omega(M,\F)$.  We say that a smooth map
$f\colon[0,1]\times M\to M'$ is a \emph{$\g$-equivariant foliate
  homotopy} if the map $f_t\colon M\to M'$ defined by $f_t(x)=f(t,x)$
is $\g$-equivariant foliate for all $t\in[0,1]$.

The following statement, the non-equivariant version of which is due
to T\"oben~\cite[\S\,2]{toeben;localization-basic-characteristic}, is
a variant of the standard homotopy lemma in de Rham theory.  A
\emph{morphism of $\g$-differential graded modules}
$\phi\colon\M_1\to\M_2$ is a degree $0$ linear map $\phi$ satisfying
$[d,\phi]=[\iota(\xi),\kappa]=[L(\xi),\kappa]=0$ for all $\xi\in\g$.
Let $\phi_0$, $\phi_1\colon\M_1\to\M_2$ be two morphisms of
$\g$-differential graded modules.  A \emph{homotopy of
  $\g$-differential graded modules} between two morphisms $\phi_0$ and
$\phi_1$ is a degree $-1$ map $\kappa\colon\M_1\to\M_2$ satisfying
$\phi_1-\phi_0=[d,\kappa]=d\kappa+\kappa d$ and
$[\iota(\xi),\kappa]=[L(\xi),\kappa]=0$ for all $\xi\in\g$.

\begin{lemma}\label{lemma;cochain-homotopy}
Let $(M,\F)$ and $(M',\F')$ be foliated manifolds equipped with
transverse actions of a Lie algebra $\g$.  Let $f\colon[0,1]\times
M\to M'$ be a $\g$-equivariant foliate homotopy.  Then the pullback
morphisms $f_0$ and $f_1\colon\Omega(M',\F')\to\Omega(M,\F)$ are
homotopic as morphisms of $\g$-differential graded algebras.  In
particular they induce the same homorphisms in equivariant basic
cohomology: $f_0^*=f_1^*\colon H_\g(M',\F')\to H_\g(M,\F)$.
\end{lemma}

\begin{proof}
As reviewed in Appendix~\ref{section;fibre}, on the ordinary de Rham
complexes we have a homotopy $f_1^*-f_0^*=[d,\kappa]$ given by the
homotopy operator $\kappa=\pi_*\circ
f^*\colon\Omega(M')\to\Omega(M)[-1]$.  Here $\pi\colon[0,1]\times M\to
M$ is the projection and $\pi_*$ denotes integration over the fibre
$[0,1]$.  We assert that the operator $\kappa$ preserves the basic
subcomplexes and is a homotopy of $\g$-differential graded algebras
when restricted to those subcomplexes.  To show this let us furnish
the cylinder $[0,1]\times M$ with the foliation $*\times\F$, whose
leaves are of the form $\{t\}\times\F(x)$ for $t\in[0,1]$ and $x\in
M$.  Then the homotopy $f$ is a foliate map.  Let $u$ be a vector
field on $M$ tangent to $\F$.  Then $(0,u)$ is a vector field on
$[0,1]\times M$ tangent to $*\times\F$.  The vector fields $u$ and
$(0,u)$ are $\pi$-related, so
Lemma~\ref{lemma;fibre}\eqref{item;fibre-derivation} tells us that
\[
\iota(u)\circ\pi_*=-\pi_*\circ\iota((0,u)),\qquad
L(u)\circ\pi_*=\pi_*\circ L((0,u)).
\]
If $u'\in\X(\F')$ is a vector field on $M'$ which is $f$-related to
$(0,u)$, then
\[
\iota((0,u))\circ f^*=f^*\circ\iota(u'),\qquad L((0,u))\circ
f^*=f^*\circ L(u').
\]
Combining these identities gives
\[
\iota(u)\circ\kappa=-\kappa\circ\iota(u'),\qquad
L(u)\circ\kappa=\kappa\circ L(u').
\]
It follows that $\kappa$ maps $\Omega(M',\F')$ to $\Omega(M,\F)$.  The
transverse $\g$-action on $(M,\F)$ extends to a transverse $\g$-action
on $([0,1]\times M,*\times\F)$ by letting $\g$ act trivially in the
$t$-direction.  Then the maps $\pi$ and $f$ are $\g$-equivariant, so
the transverse vector fields $\xi_M$, $\xi_{[0,1]\times M}$ and
$\xi_{M'}$ are $\pi$- and $f$-related.  Again by
Lemma~\ref{lemma;fibre}\eqref{item;fibre-derivation}, $\pi_*$ commutes
with the operations $\iota(\xi)$ and $L(\xi)$, and so does $f^*$.
This shows that $\kappa\colon\Omega(M',\F')\to\Omega(M,\F)[-1]$ is a
homotopy of $\g$-differential graded modules.
\end{proof}

\subsection{Equivariant basic characteristic forms}
\label{subsection;characteristic}

Let $K$ be a Lie group with Lie algebra $\liek$ and let
$\pi\colon(P,\F_P)\to(M,\F=\F_M)$ be a foliated principal $K$-bundle
as in \S~\ref{subsection;bundle}.  For the purpose of constructing
Thom classes in \S~\ref{subsection;existence} we recall how under
appropriate hypotheses one can define characteristic classes of $P$ in
the basic cohomology of $M$ and how, in the presence of a transverse
action of a Lie algebra, these classes lift to equivariant basic
classes.

A connection on the principal bundle $P$ can be viewed either as a
bundle map $\pi^*TM\to TP$ or as a $1$-form
$\theta\in\Omega^1(P,\liek)$.  A connection on $P$ is called
\emph{basic} (or \emph{projectable}
in~\cite{molino;riemannian-foliations}) if the associated $1$-form is
$\F_P$-basic.  According
to~\cite[\S\,2.6]{molino;riemannian-foliations} a connection $1$-form
$\theta$ is $\F_P$-basic if and only if for every foliation chart $U$
of $M$ with quotient manifold $\bar{U}=U/(\F|_U)$ the restriction of
$\theta$ to $P|_U$ is the pullback of a connection $1$-form
$\bar\theta_U\in\Omega^1(\bar{U},\bar{P}_U)$ on the quotient principal
bundle $\bar{P}_U=P|_U\big/\bigl(\F_P|_{\pi^{-1}(U)}\bigr)$ over
$\bar{U}$.  In particular a basic connection on $P$, viewed as a
bundle map $\pi^*TM\to TP$, is an extension of the partial connection
$\pi^*T\F_M\to T\F_P$ defined in~\eqref{equation;partial-connection}.

Now suppose $P$ is equipped with a transverse action of a
finite-dimensional Lie algebra $\g$ which commutes with the
$K$-action.  This transverse action descends to a unique transverse
action on the base manifold $M$ with the property that the bundle
projection $\pi\colon P\to M$ is $\g$-equivariant.  We say that $P$ is
a \emph{$\g$-equivariant foliated principal $K$-bundle}.

A natural instance of such a bundle is the transverse frame bundle $P$
of the foliation $\F$.  If $a\colon\g\to\X(M,\F)$ is any transverse
$\g$-action on $M$, the homomorphism $\pi^\dag\circ
a\colon\g\to\X(P,\F_P)^K$, where $\pi^\dag$ is the lifting
homomorphism~\eqref{equation;natural-lifting}, makes $P$ a
$\g$-equivariant foliated principal $K$-bundle.

\begin{lemma}\label{lemma;basic-connection}
Suppose that the $\g$-equivariant foliated principal $K$-bundle $P$
admits a principal connection $\pi^*TM\to TP$ which is $\g$-invariant
and $\F_P$-basic.  Then the associated connection $1$-form
$\theta\in\Omega^1(P,\F_P;\liek)$, viewed as a map
$\liek^*\to\Omega^1(P,\F_P)$ defines a $\g$-invariant connection on
the $\liek$-differential graded algebra $\Omega(P,\F_P)$.
\end{lemma}

\begin{proof}
This is a restatement of the definition of $\g$-invariant connections
on $\liek$-differential graded algebras; see
\S\S\,\ref{section;gtilde-algebras} and~\ref{section;change}.
\end{proof}

Under the hypothesis of Lemma~\ref{lemma;basic-connection} we have the
$\g$-equivariant characteristic homomorphism defined
in~\eqref{equation;characteristic},
\[
c_{\g,\theta}\colon S(\liek[2]^*)^{\liek}\longto
\bigl(\Omega_\bbas{\liek}(P,\F_P)\bigr)_\g.
\]
Here the left-hand side denotes the algebra of $\liek$-invariant
polynomials on $\liek^*$ with the generators placed in degree $2$.
The right-hand side is the $\g$-Weil complex
$\bigl(\W\g\otimes\Omega_\bbas{\liek}(P,\F_P)\bigr)_\bbas{\g}$ of the
$\liek$-basic $\F_P$-basic de Rham complex of $P$.  If moreover the
structure group $K$ is connected, then
$\Omega_\bbas{\liek}(P,\F_P)=\Omega_\bbas{K}(P,\F_P)=\Omega(M,\F)$ as
a $\g$-differential graded algebra, so
$\bigl(\Omega_\bbas{\liek}(P,\F_P)\bigr)_\g=\Omega_\g(M,\F)$.  So if
$K$ is connected, the $\g$-equivariant characteristic homomorphism is
an algebra homomorphism
\begin{equation}\label{equation;equivariant-basic-characteristic}
c_{\g,\theta}\colon S(\liek[2]^*)^{\liek}\longto\Omega_\g(M,\F).
\end{equation}
Elements in the image
of~\eqref{equation;equivariant-basic-characteristic} are
\emph{$\g$-equivariant $\F$-basic characteristic forms}; their
cohomology classes are \emph{$\g$-equivariant $\F$-basic
  characteristic classes}.

If $\g=0$, the characteristic map is an algebra homomorphism
$c_\theta\colon S(\liek[2]^*)^{\liek}\to\Omega(M,\F)$ and we speak of
$\F$-basic characteristic forms and classes.

If the leaf space $\bar{M}=M/\F$ is a manifold and the principal
bundle $P$ descends to a principal bundle $\bar{P}$ on $\bar{M}$ as in
Remark~\ref{remark;quotient}, then the $\g$-equivariant $\F$-basic
characteristic forms of $P$ are the same as the $\g$-equivariant
characteristic forms of the quotient bundle $\bar{P}$.

An obstruction to the existence of basic connections on foliated
principal bundles known as the \emph{Atiyah class}, and examples of
bundles whose Atiyah class does not vanish, can be found
in~\cite[Ch.~8]{kamber-tondeur;foliated-bundles}
and~\cite[Ch.~2]{molino;riemannian-foliations}.  We are not aware of
any references for a $\g$-equivariant version of this obstruction.

\subsection{The equivariant basic Thom isomorphism}

Let $\g$ be a finite-dimen\-sional real Lie algebra and let $M$ be a
foliated $\g$-manifold.  Let $E$ be an oriented $\g$-equivariant
foliated vector bundle over $M$.  An \emph{equivariant basic Thom
  form} of $E$ is an $r$-form $\tau_\g\in\Omega_{\g,\cv}^r(E,\F_E)$
which satisfies $\pi_*\tau_\g=1$ and $d_\g\tau_\g=0$.  Such a form
does not necessarily exist.  An example where it does not exist is
given in \S~\ref{subsection;without}, and sufficient conditions for it
to exist are stated in Proposition~\ref{proposition;thom-form}.  If it
exists we have the following $\g$-equivariant $\F$-basic Thom
isomorphism theorem.

\begin{theorem}\label{theorem;equivariant-basic-thom}
Suppose that $E$ admits an equivariant basic Thom form $\tau_\g$.
Then the following conclusions hold.
\begin{enumerate}
\item\label{item;equivariant-basic-homotopy-equivalence}
Fibre integration
\[
\pi_*\colon\Omega_{\g,\cv}(E,\F_E)[r]\longto\Omega_\g(M,\F)
\]
is a homotopy equivalence.  A homotopy inverse of $\pi_*$ is the Thom
map
\[
\zeta_*\colon\Omega_\g(M,\F)\longto\Omega_{\g,\cv}(E,\F_E)[r]
\]
defined by $\zeta_*(\alpha)=\tau_\g\wedge\pi^*\alpha$.  A homotopy
$\zeta_*\pi_*\simeq\id$ is induced by the homotopy of
$\g$-differential graded modules given
in~\eqref{equation;g-homotopy-formula} below.
\item
All equivariant basic Thom forms of $E$ are cohomologous.  Their
cohomology class $\Thom_\g(E,\F_E)\in H_{\g,\cv}^r(E,\F_E)$ is
uniquely determined by the property $\pi_*(\Thom_\g(E,\F_E))=1$.
\item
$H_{\g,\cv}(E,\F_E)$ is a free $H_\g(M,\F)$-module of rank $1$
  generated by the Thom class $\Thom_\g(E,\F_E)$.
\end{enumerate}
\end{theorem}

\begin{proof}
We will verify that each step in the proof of
Theorem~\ref{theorem;thom} is valid in the present context.  Only the
proof of~\eqref{item;equivariant-basic-homotopy-equivalence} requires
comment.  We start by showing that the fibre integral of an
equivariant basic form is an equivariant basic form.  The usual fibre
integration map $\pi_*\colon\Omega_\cv(E)[r]\to\Omega(M)$ has the
properties
\[
d\pi_*=(-1)^r\pi_*d,\qquad
L(v)\pi_*=\pi_*L(w),\qquad\iota(v)\pi_*=(-1)^r\pi_*\iota(w)
\]
for every pair of $\pi$-related vector fields $v\in\X(M)$ and
$w\in\X(E)$.  (See Appendix~\ref{section;fibre}.)  It follows from
these properties that $\pi_*$ maps $\F_E$-basic forms to $\F$-basic
forms and restricts to a degree $-r$ morphism of $\g$-differential
graded modules
\begin{equation}\label{equation;fibre-basic}
\pi_*\colon\Omega_\cv(E,\F_E)[r]\longto\Omega(M,\F),
\end{equation}
where $\Omega_\cv(E,\F_E)=\Omega_\cv(E)\cap\Omega(E,\F)$ denotes the
$\g$-differential graded module of vertically compactly supported
basic forms on $E$.  The morphism~\eqref{equation;fibre-basic} extends
uniquely to an $(S\g^*)^\g$-linear degree $-r$ morphism of Weil
complexes
\[
\pi_*\colon\Omega_{\g,\cv}(E,\F_E)[r]\longto\Omega_\g(M,\F),
\]
and the projection formula~\eqref{equation;projection} shows that this
map is a degree $-r$ morphism of graded left
$\Omega_\g(M,\F)$-modules.  Now let $\zeta_*$ be the Thom map defined
by the equivariant basic Thom form $\tau_\g$.  The identity
$\pi_*\zeta_*=\id$ holds as in the proof of
Theorem~\ref{theorem;thom}.  Next we must find a cochain homotopy
$\kappa_\g$ of the complex $(\Omega_{\g,\cv}(E,\F_E),d_\g)$ satisfying
\[\zeta_*\pi_*-{\id}=d_\g\kappa_\g+\kappa_\g d_\g.\]
By definition $\Omega_{\g,\cv}(E,\F_E)$ is the Weil complex
$\M_\g=(\W\g\otimes\M)_\bbas{\g}$ of the $\g$-differential graded
module $\M=\Omega_\cv(E,\F_E)$, so it is enough to find a homotopy
$\kappa$ of the $\g$-differential graded module $\W\g\otimes\M$ with
the property
\begin{equation}\label{equation;foliated-homotopy}
\zeta_*\pi_*-{\id}=d\kappa+\kappa d.
\end{equation}
Replacing $\tau$ with $\tau_\g$ in~\eqref{equation;homotopy-formula}
we put
\begin{equation}\label{equation;g-homotopy-formula}
\kappa=(-1)^rp_{1,*}\circ p_*\circ\varrho^*\circ l(p_2^*\tau_\g)\circ
p_1^*,
\end{equation}
where the maps
\begin{gather*}
p_1,p_2\colon E\oplus E\to E,\qquad p\colon[0,1]\times(E\oplus E)\to
E\oplus E,\\
\varrho\colon[0,1]\times(E\oplus E)\to E\oplus E
\end{gather*}
are as in the proof of Theorem~\ref{theorem;thom}.  Then $\kappa$
satisfies~\eqref{equation;foliated-homotopy} and maps vertically
compactly supported forms to vertically compactly supported forms just
as before.  The direct sum $E\oplus E$ is a foliated vector bundle
with foliation $\F_{E\oplus E}$, and the cylinder $[0,1]\times(E\oplus
E)$ carries the foliation $*\times\F_{E\oplus E}$.  With respect to
these foliations the maps $p_1$, $p_2$, $\varrho$ and $p$ are foliate,
and the Thom form $\tau_\g$ is $\F_E$-basic, so $\kappa$ maps basic
forms to basic forms.  Moreover, the maps $p_1$, $p_2$, $\varrho$ and
$p$ are $\g$-equivariant and the Thom form $\tau_\g$, regarded an
element of $\M_\g\subset\W\g\otimes\M$, is $\g$-basic, so the map
$\kappa$ commutes with the contractions $\iota(\xi)$ and the
derivations $L(\xi)$ for all $\xi\in\g$.  This shows that $\kappa$ is
a homotopy of the $\g$-differential graded module $\W\g\otimes\M$.
\end{proof}  

\subsection{A foliated vector bundle without basic Thom form}
\label{subsection;without}

Not all foliated vector bundles have basic Thom forms.  As an example
consider the $2$-torus, i.e.\ the product $M=\T_1\times\T_2$ of two
copies $\T_1$, $\T_2$ of the circle $\T=\R/\Z$.  Let $\F=\F_M$ be a
foliation of $M$ with the following properties:
\begin{enumerate}
\renewcommand\theenumi{\alph{enumi}}
\renewcommand\labelenumi{(\theenumi)}
\item\label{item;action}
the $\T$-action
$t\cdot(t_1,t_2)=(t_1,tt_2)$ on $M$ is foliate, i.e.\ maps leaves to
leaves;
\item\label{item;orbit-leaf}
the $\T$-orbit $\ca{L}_0=\{0\}\times\T_2$ is a leaf;
\item\label{item;leaf-closure}
the leaf $\ca{L}_0$ is in the closure of every other leaf;
\item\label{item;leaf-transverse}
every leaf other than $\ca{L}_0$ is transverse to the $\T$-orbits.
\end{enumerate}
An example of such a foliation is produced at the end
of~\cite[\S\,1.4]{molino;riemannian-foliations}.



\begin{figure}[hb] 
\centering\includegraphics[width=6cm]{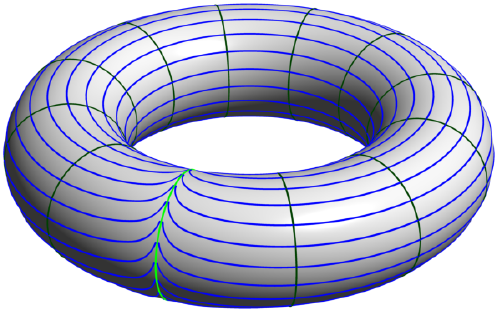}
\qquad\qquad\includegraphics[height=3.5cm,width=3.5cm]{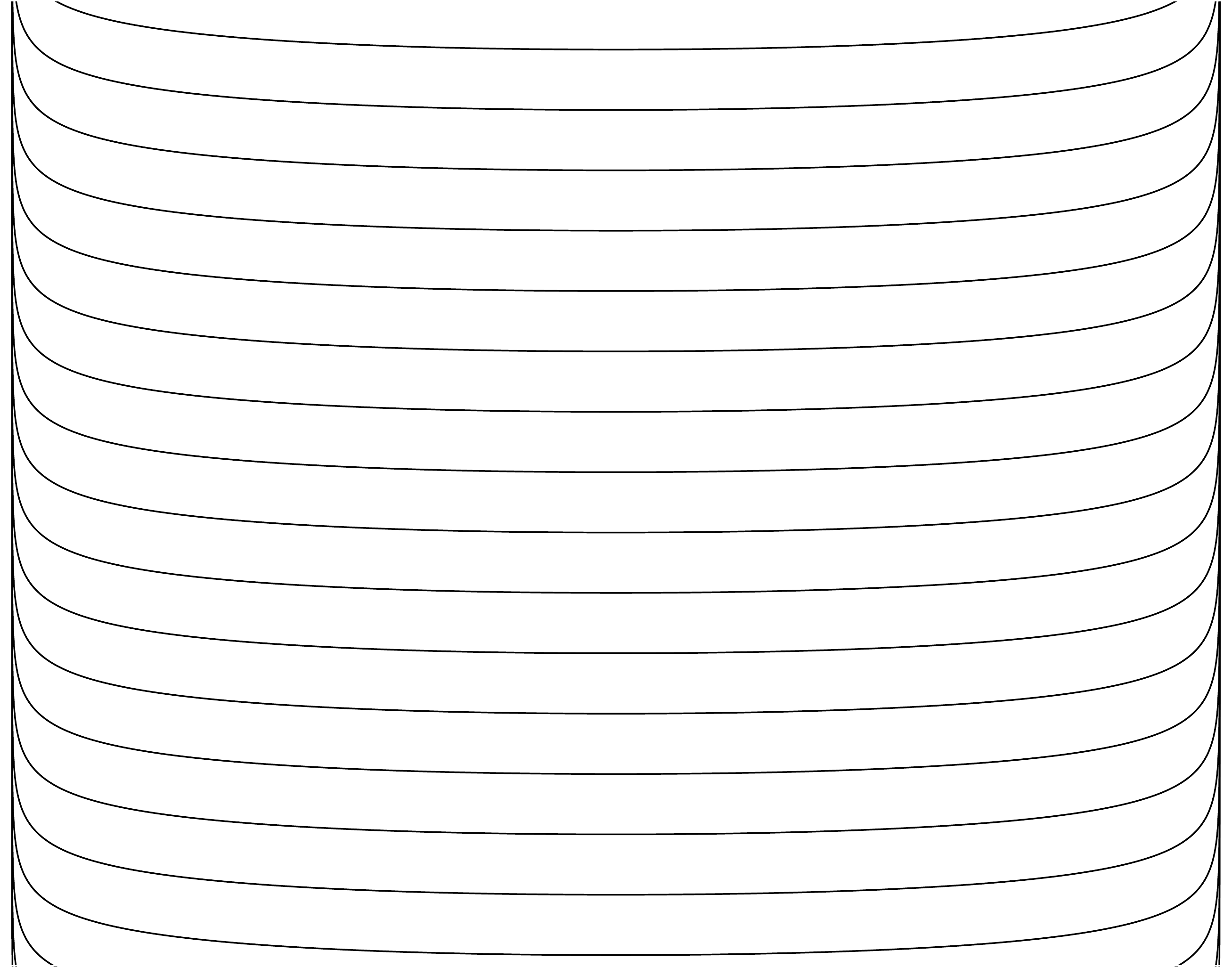}
\caption{Foliation of $2$-torus.  Compact leaf $\ca{L}_0$ shown in
  bright green, noncompact leaves in blue.  Orbits of $\T$-action
  shown in green.  Lifted foliation of unit square shown on right.}
\end{figure}

The normal bundle $E=N\F=TM/T\F$ of the foliation $\F$ is a foliated
vector bundle over $M$ of rank~$1$.  It is also an equivariant vector
bundle with respect to the $\T$-action on $M$.  The foliate vector
field $\tilde{v}=\partial/\partial t_2$ is the generator of the
$\T$-action on $M$.  Let $v=\tilde{v}\bmod\X(\F)\in\Gamma(E)$ be the
transverse vector field determined by $\tilde{v}$.  The actions of $v$
on $M$ and $E$ make $E$ a $\g$-equivariant foliated vector bundle,
where $\g=\R$.

\begin{lemma}\phantomsection\label{lemma;fmfebasic}
\begin{enumerate}
\item\label{item;fmbasic}
$\Omega^0(M,\F)=\R$ and $\Omega^1(M,\F)=0$.
\item\label{item;febasic}
$\Omega^0(E,\F_E)=\R$. 
\end{enumerate}
\end{lemma}

\begin{proof}
\eqref{item;fmbasic}~This is verified as in \cite[pp.~17,
  40]{molino;riemannian-foliations}.

\eqref{item;febasic}~Let $f$ be an $\F_E$-basic function on $E$.
Viewed as a map $M\to E$ the section $v$ of $E$ is foliate, and
therefore the function $f\circ v$ is $\F$-basic on $M$.  More
generally, define a map $u\colon\R\times M\to E$ by
\[u(\lambda,x)=u_\lambda(x)=\lambda v(x).\]
For each $\lambda\in\R$ the scalar multiple $u_\lambda=\lambda v$ of
$v$ is a foliate map $M\to E$, and therefore for each $\lambda$ the
function $f\circ u_\lambda$ is $\F$-basic on $M$.  Thus $f\circ
u_\lambda$ is constant by~\eqref{item;fmbasic}.  Since $\tilde{v}$ is
tangent to the compact leaf $\ca{L}_0$ (by
property~\eqref{item;orbit-leaf} above), we have $v=0$ on $\ca{L}_0$
and so $u_\lambda(x_0)=\lambda v(x_0)=0$ for all $x_0\in\ca{L}_0$.
Hence $f(u_\lambda(x))=f(u_\lambda(x_0))=f(x_0)=c$ for all $x\in M$,
where $c$ is the (constant) value of $f$ on $\ca{L}_0$.  (Here we have
for simplicity identified $M$ with the zero section of $E$.)  In other
words, $f$ is constant on the image of the map $u$.  The transverse
vector field $v$ vanishes nowhere on $M\backslash\ca{L}_0$ (by
property~\eqref{item;leaf-transverse} above) and $E$ is of rank $1$,
so the image of the map $u$ contains $E\backslash\ca{L}_0$, which is
dense in $E$.  We conclude that $f=c$ on $E$.
\end{proof}

\begin{proposition}
If $\alpha\in\Omega^1(E,\F_E)$ is closed, then $\alpha=0$.  It follows
that $H_\cv^1(E,\F_E)=0$.  Therefore the foliated vector bundle $E$
has no basic Thom form, nor does $E$ have an equivariant basic Thom
form with respect to any transverse Lie algebra action on $E$ (such as
the action generated by the transverse vector field $v$).
\end{proposition}

\begin{proof}
The retraction along the fibres is a foliate homotopy equivalence
between $E$ and $M$, so $H^1(E,\F_E)\cong H^1(M,\F)$ by
Lemma~\ref{lemma;cochain-homotopy}.  Hence $H^1(E,\F_E)=0$ by
Lemma~\ref{lemma;fmfebasic}\eqref{item;fmbasic}.  Therefore
$\alpha=df$ for some basic function $f\in\Omega^0(E,\F_E)$.  But $f$
is constant by Lemma~\ref{lemma;fmfebasic}\eqref{item;febasic}, so
$\alpha=0$.  In particular $H_\cv^1(E,\F_E)=0$.  On the other hand
$H^0(M,\F)=\R\ne H_\cv^1(E,\F_E)$ by
Lemma~\ref{lemma;fmfebasic}\eqref{item;fmbasic}, so by
Theorem~\ref{theorem;equivariant-basic-thom} there cannot exist a
basic Thom form for $E$.  If a Lie algebra $\g$ acts transversely on
$E$, then the image of an equivariant basic Thom form $\tau_\g$ under
the forgetful homomorphism $\Omega_\g(E,\F_E)\to\Omega(E,\F_E)$ would
be an ordinary basic Thom form, so $\tau_\g$ does not exist either.
\end{proof}

\subsection{Existence of Thom forms}\label{subsection;existence}

The main result of this section is
Proposition~\ref{proposition;thom-form}, which extends results
of~\cite[Appendix~A]%
{goertsches-nozawa-toeben;chern-simons-foliations}, and which gives a
sufficient condition for an oriented $\g$-equivariant foliated vector
bundle $E$ to possess an equivariant basic Thom form.  The condition
is that the structure group of the foliated bundle $E$ should admit a
Riemannian metric compatible with the foliation, that the Lie algebra
should act isometrically, and that $E$ should admit a $\g$-invariant
basic metric connection.

Let $(M,\F)$ be a foliated manifold and $(E,\F_E)$ a oriented foliated
vector bundle over $M$.  A \emph{Riemannian metric} on $E$ is a
Riemannian fibre metric $g_E$ which satisfies $\nabla_{E,v}g_E=0$ for
all $v\in\X(\F)$, where $\nabla_E$ is the partial connection of
$(E,\F_E)$.  If $\g$ is a finite-dimensional Lie algebra acting
transversely on $M$, the vector bundle $E$ is $\g$-equivariant, and
the metric $g_E$ is $\g$-invariant, we say $(E,\F_E,g_E)$ is a
\emph{$\g$-equivariant Riemannian foliated vector bundle}.

\begin{proposition}\label{proposition;thom-form}
Let $(M,\F,a)$ be foliated $\g$-manifold and $(E,\F_E,g_E)$ an
oriented $\g$-equivariant Riemannian foliated vector bundle over $M$.
Suppose that the oriented orthogonal frame bundle $P$ of $E$ admits a
connection that is $\g$-invariant and $\F_P$-basic as in
Lemma~\ref{lemma;basic-connection}.  Then there exists an equivariant
basic Thom form on $E$, and hence the Thom isomorphism theorem,
Theorem~\ref{theorem;equivariant-basic-thom}, applies to $E$.
\end{proposition}

A very special case where the hypotheses of this proposition are
satisfied is when the foliations $\F_M$ and $\F_E$ are $0$-dimensional
and the $\g$-actions on $M$ and $E$ are induced by actions of a
compact Lie group $G$ with Lie algebra $\g$.  In this case the
existence of the invariant metric on $E$ is automatic and the
proposition gives the existence of $G$-equivariant Thom forms, which
is well known and can be found
e.g.\ in~\cite[\S\,4.5]{paradan-vergne;thom-chern}.  Another case
where the hypotheses are satisfied is when $E$ is the normal bundle of
$M$ in an ambient Riemannian foliated $\g$-manifold; we will
investigate this case in \S\,\ref{section;foliated-thom-gysin}.
Proposition~\ref{proposition;thom-form} fails for the foliation of
\S\,\ref{subsection;without} because that foliation is not Riemannian.

Proposition~\ref{proposition;thom-form} is an immediate consequence of
Lemma~\ref{lemma;thom} below, which exhibits a specific equivariant
basic Thom form on $E$, called universal.  Suppose that the conditions
of Proposition~\ref{proposition;thom-form} are satisfied.  Let $K$ be
the special orthogonal group $\SO(r)$ and $\liek=\lie{o}(r)$ its Lie
algebra.  Let $g_E$ be a fibre metric as in
Proposition~\ref{proposition;thom-form} and let
$\theta\in\Omega^1(P,\liek)$ be a $\g$-invariant $\F_P$-basic
connection on the oriented orthogonal frame bundle $P$ of the
$\g$-equivariant foliated vector bundle $E\to M$.  Let
$\tau_0\in\Omega_{\liek,\cc}^r(\R^r)$ be a compactly supported
modification of the Mathai-Quillen-Thom form on $\R^r$ as defined
in~\cite[p.~98]{mathai-quillen;thom-equivariant};
cf.\ also~\cite[\S\,10.3]{guillemin-sternberg;supersymmetry-equivariant}
or~\cite[\S\,8]{meinrenken;equivariant-cohomology-cartan}.  The
subscript ``$\cc$'' refers to compact supports; we regard $\tau_0$ as
a $\liek$-basic element of the $\liek$-differential graded module
$\W\liek\otimes\Omega_\cc(\R^r)$.  We view this module as a submodule
of the $\h$-differential graded module $\W\h\otimes\Omega_\cc(\R^r)$,
where $\h$ is the product Lie algebra $\liek\times\g$ and where we let
$\g$ act trivially on $\R^r$.  The frame bundle $P$ is a
$\g$-equivariant foliated principal $K$-bundle over $M$ with foliation
$\F_P$.  The product $P\times\R^r$ has a foliation $\F_P\times*$,
where $*$ denotes the $0$-dimensional foliation of $\R^r$.  The
$K$-action and the transverse $\g$-action provide the complex
$\M=\Omega_\cv(P\times\R^r,\F_P\times*)$ with the structure of an
$\h$-differential graded module.  Let $\pr_2$ be the projection onto
the second factor $P\times\R^r\to\R^r$.  Since $\tau_0$ is
$\liek$-basic, the pullback $\pr_2^*(\tau_0)$ is an $\h$-basic element
of $\W\h\otimes\M$.  Let $C_{\g,\theta}$ be the $\g$-equivariant
Cartan-Chern-Weil homomorphism associated with the $\g$-invariant
connection $\theta$, as defined in \S\,\ref{section;change}.  Then
$C_{\g,\theta}(\pr_2^*(\tau_0))$ is an $\h$-basic element of
$\W\g\otimes\M$.  We summarize the situation with the diagram
\[
\begin{tikzcd}
\tau_0\in\W\liek\otimes\Omega_\cc(\R^r)\ar[r,hook]&
\W\h\otimes\Omega_\cc(\R^r)\ar[r,"\pr_2^*"]&
\W\h\otimes\M\ar[r,"C_{\g,\theta}"]&\W\g\otimes\M.
\end{tikzcd}
\]
The foliated vector bundle $E$ is the quotient of $P\times\R^r$ by the
free diagonal $K$-action, so the $\liek$-basic subcomplex of $\M$ is
\begin{align*}
\M_\bbas{\liek}&=\Omega_{\bbas{\liek},\cv}(P\times\R^r,\F_P\times*)=
\Omega_{\bbas{K},\cv}(P\times\R^r,\F_P\times*)\\
&=\Omega_\cv(E,\F_E).
\end{align*}
Therefore the $\h$-basic subcomplex of $\W\g\otimes\M$ is
\[
(\W\g\otimes\M)_{\bbas{\h}}=(\W\g\otimes\M_{\bbas{\liek}})_{\bbas{\g}}=
\Omega_{\g,\cv}(E,\F_E).
\]
We define the \emph{universal equivariant basic Thom form} of
$(E,\theta)$ to be the image
\[
\tau_{\g,\theta}(E,\F_E)=C_{\g,\theta}(\pr_2^*(\tau_0))\in
\Omega_{\g,\cv}^r(E,\F_E).
\]
In part~\eqref{item;chern-gauss-bonnet} of the next result, $\Pf\in
S^{r/2}(\liek^*)^\liek$ denotes the Pfaffian (which is defined when
$r$ is even), and the map $c_{\g,\theta}\colon
S(\liek^*)^{\liek}\to\Omega_\g^*(M,\F)$ denotes the $\g$-equivariant
characteristic
homomorphism~\eqref{equation;equivariant-basic-characteristic} of the
foliated bundle $E$ with respect to the invariant basic orthogonal
connection $\theta$.

\begin{lemma}\label{lemma;thom}
Let $(M,\F,a)$ be foliated $\g$-manifold and $(E,\F_E,g_E)$ an
oriented $\g$-equivariant Riemannian foliated vector bundle over $M$.
Let $\theta\in\Omega^1(P,\liek)$ be an invariant basic connection on
the oriented orthogonal frame bundle $P$ of $E$.  The form
$\tau_{\g,\theta}(E,\F_E)$ has the following properties.
\begin{enumerate}
\item\label{item;universal-thom}
$\tau_{\g,\theta}(E,\F_E)$ is an equivariant basic Thom form.
\item\label{item;universal-universal}
$\tau_{\g,\theta}(E,\F_E)$ is universal in the sense that
  $\tau_{\g,f^*\theta}(f^*E,f^*\F_E)=f_E^*\tau_{\g,\theta}(E,\F_E)$
  for all $\g$-equivariant foliate maps $f\colon(M',\F')\to(M,\F)$,
  where $f_E\colon f^*E\to E$ is the natural lift of $f$.
\item\label{item;chern-gauss-bonnet}
$\zeta^*\tau_{\g,\theta}(E,\F_E)=0$ if $r$ is odd and
  $\zeta^*\tau_{\g,\theta}(E,\F_E)=(-2\pi)^{-r/2}c_{\g,\theta}(\Pf)$
  if $r$ is even.
\item\label{item;universal-basic}
Let $p\colon P\times\R^r\to E$ be the quotient map for the $K$-action.
Let $\ca{E}$ be any (possibly singular) foliation of $E$ with the
property that every vector field tangent to $\ca{E}$ is $p$-related to
a vector field on $P\times\R^r$ which is tangent to the fibres of
$\pr_2$.  Then $\tau_{\g,\theta}(E,\F_E)$ is basic with respect to
$\ca{E}$.
\end{enumerate}
\end{lemma}

\begin{proof}
\eqref{item;universal-thom}~Put $\tau=\tau_{\g,\theta}(E,\F_E)$.  We
have $\int_{E_x}\tau=\int_{\R^r}\tau_0=1$ for all $x\in M$, so
$\pi_*\tau=1$.  Also $d_\g\tau=0$ because $d_{\liek}\tau_0=0$ and the
Cartan-Chern-Weil map is a cochain map.

\eqref{item;universal-universal}~This follows from the naturality of
the Cartan-Chern-Weil map with respect to maps.

\eqref{item;chern-gauss-bonnet}~We have $\zeta^*\tau=j^*\tau_0$, where
$j\colon0\to\R^r$ is the inclusion of the origin.  The assertion now
follows from the fact that $j^*\tau_0=0$ if $r$ is odd and
$j^*\tau_0=(-2\pi)^{-r/2}c_{\g,\theta}(\Pf)$ if $r$ is even;
see~\cite[\S\,7]{mathai-quillen;thom-equivariant}
or~\cite[(7.20)]{guillemin-sternberg;supersymmetry-equivariant}.

\eqref{item;universal-basic}~This follows from the fact that
$\pr_2^*\tau_0$ is basic with respect to the projection $\pr_2\colon
P\times\R^r\to\R^r$.
\end{proof}

\subsection{Euler forms}\label{section;euler}

Let $\g$ be a finite-dimensional Lie algebra, $(M,\F,a)$ foliated
$\g$-manifold, and $(E,\F_E,g_E)$ an oriented $\g$-equivariant
Riemannian foliated vector bundle of rank $r$ over $M$.  Suppose that
the oriented orthogonal frame bundle of $E$ admits an invariant basic
connection $\theta$ as in Proposition~\ref{proposition;thom-form}.
Then we have a $\g$-equivariant characteristic homomorphism
$c_{\g,\theta}$ as in
\eqref{equation;equivariant-basic-characteristic}.  The
\emph{universal $\g$-equivariant $\F$-basic Euler form} of the
foliated Riemannian vector bundle $E$ is the element
$\eta_{\g,\theta}(E,\F_E)\in\Omega_\g^r(M,\F)$ given by
\[
\eta_{\g,\theta}(E,\F_E)=
\begin{cases}
0&\text{if $r$ is odd}\\
(-2\pi)^{-r/2}c_{\g,\theta}(\Pf)&\text{if $r$ is even}.
\end{cases}
\]
The next statement follows immediately from
Theorem~\ref{theorem;equivariant-basic-thom} and
Lemma~\ref{lemma;thom}.

\begin{proposition}\label{proposition;euler}
Let $(M,\F,a)$ be foliated $\g$-manifold and $(E,\F_E,g_E)$ an
oriented $\g$-equivariant Riemannian foliated vector bundle of rank
$r$ over $M$.  Suppose that $E$ admits an invariant basic metric
connection.  The universal equivariant basic Euler form satisfies
\[
\eta_{\g,\theta}(E,\F_E)=\zeta^*(\tau_{\g,\theta}(E,\F_E))=
\zeta^*\zeta_*(1),
\]
where $\tau_{\g,\theta}(E,\F_E)$ is the universal equivariant basic
Thom form of $E$, and
\[\eta_{\g,f^*\theta}(f^*E,f^*\F_E)=f^*\eta_{\g,\theta}(E,\F_E)\]
for all $\g$-equivariant foliate maps $f\colon(M',\F')\to(M,\F)$.
\end{proposition}

\section{Riemannian foliations and normal bundles}
\label{section;foliated-thom-gysin}

In this section $\g$ denotes a finite-dimensional real Lie algebra and
$(X,\F)$ a foliated $\g$-manifold as defined in
\S~\ref{section;foliated-thom}.  We let $i\colon Y\to X$ be a
co-oriented closed submanifold which is preserved by the $\g$-action.
The normal bundle $NY=i^*TX/TY$ is an oriented $\g$-equivariant
foliated vector bundle over $Y$.  The first goal of this section is to
establish Proposition~\ref{proposition;normal-thom}, which says that
$NY$ admits an equivariant basic Thom form, provided that the
foliation $\F$ is Riemannian and the $\g$-action is isometric (the
definition of which we recall below).  This fact enables us to define
a wrong-way homomorphism
$i_*\colon\Omega_\g(Y,\F_Y)\to\Omega_\g(X,\F)$ and to obtain a long
exact Thom-Gysin sequence, which relates the equivariant basic
cohomology of $X$ to that of $Y$ and the complement $X\backslash Y$
(Proposition%
~\ref{proposition;thom-gysin-submanifold-equivariant-basic}).

\subsection{Extending and reducing connections}

We state without proof two elementary facts regarding principal
connections.  Suppose we are given a Lie group homomorphism
$f=f_G\colon G_1\to G_2$, a principal $G_1$-bundle $P_1\to X_1$, a
principal $G_2$-bundle $P_2\to X_2$, and a smooth map $f=f_P\colon
P_1\to P_2$ with the equivariance property $f_P(gp)=f_G(g)f_P(p)$ for
all $g\in G_1$ and $p\in P_1$.  Then $f_G$ induces a Lie algebra
homomorphism $f=f_\g\colon\g_1\to\g_2$ and $f_P$ descends to a smooth
map $f=f_X\colon X_1\to X_2$ as in the commutative diagram
\[
\begin{tikzcd}
P_1\ar[r,"f_P"]\ar[d]&P_2\ar[d]
\\
X_1\ar[r,"f_X"]&X_2.
\end{tikzcd}
\]

\begin{lemma}\phantomsection\label{lemma;connection}
\begin{enumerate}
\item\label{item;extend}
Let $X_1=X_2$ and $f_X=\id_{X_1}$.  For every connection
$\theta_1\in\Omega^1(P_1,\g_1)^{G_1}$ on $P_1$ there is a unique
connection $\theta_2\in\Omega^1(P_2,\g_2)^{G_2}$ on $P_2$ with the
property that $Tf_P$ maps $\theta_1$-horizontal subspaces to
$\theta_2$-horizontal subspaces.
\item\label{item;project}
Suppose there exists an $\Ad(G_1)$-equivariant linear map
$\pr\colon\g_2\to\g_1$ satisfying ${\pr}\circ f_\g=\id_{\g_1}$.  For
every connection $\theta_2$ on $P_2$ the formula
$\theta_1={\pr}\circ\theta_2\circ Tf_P$ defines a connection
$\theta_1$ on $P_1$.
\end{enumerate}
\end{lemma}

\subsection{Thom form of the normal bundle}

Recall that a \emph{Riemannian structure} on the foliated manifold
$(X,\F)$ is a Riemannian metric $g$ on the normal bundle of the
foliation $N\F=TM/T\F$ that satisfies $\nabla g=0$, where $\nabla$ is
the partial connection on $N\F$.  A \emph{Killing vector field} on
$(X,\F)$ is a vector field $v\in\X(X)$ that satisfies $L(v)g=0$.
By~\cite[Lemma 3.5]{molino;riemannian-foliations} a Killing vector
field $v\in\X(X)$ is automatically foliate and therefore the Killing
vector fields form a Lie subalgebra $\lie{N}(\F,g)$ of the Lie algebra
of foliate vector fields $\lie{N}(\F)$.  Vector fields in $\X(\F)$ are
by definition Killing, so $\X(\F)$ is an ideal of $\lie{N}(\F,g)$.
The quotient $\X(X,\F,g)=\lie{N}(\F,g)/\X(\F)\subseteq\X(X,\F)$ is a
Lie algebra, whose elements we call \emph{transverse Killing vector
  fields}.  A transverse action $a\colon\g\to\X(X,\F)$ is
\emph{isometric} if $a(\xi)$ is transverse Killing for all $\xi\in\g$.

Fix a Riemannian structure $g$ on $(X,\F)$ and an isometric transverse
$\g$-action $a\colon\g\to\X(X,\F,g)$ on $X$.  We write $a(\xi)=\xi_X$
for $\xi\in\g$.  Let $n$ be the codimension of $\F$ and let $P\to X$
be the orthonormal frame bundle of $N\F$, which has structure group
$K=\group{O}(n)$.  By
~\cite[\S\,2.35]{kamber-tondeur;foliated-bundles}
or~\cite[\S\,3.3]{molino;riemannian-foliations}, $P$ is a foliated
bundle with foliation $\F_P$, whose partial connection extends to a
unique torsion-free basic connection $\theta\in\Omega^1(P,\liek)^K$ on
$P$, called the \emph{transverse Levi Civita connection}.  Killing
vector fields on $X$ lift naturally to foliate vector fields on $P$,
which preserve $\theta$.  Thus $\theta$ is a $\g$-invariant
$\F_P$-basic connection.

A submanifold $Y$ of $X$ is \emph{$\g$-invariant} if for every
$\xi\in\g$ and every foliate representative
$\tilde{\xi}_X\in\lie{N}(\F,g)$ of $\xi_X$ the flow of $\tilde{\xi}_X$
preserves $Y$.  Let $Y$ be $\g$-invariant.  Then in particular for
each $x\in Y$ the leaf $\F(x)$ is contained in $Y$.  Let $\F_Y$ be the
induced foliation of $Y$.  The restriction of the normal bundle $N\F$
to $Y$ is an orthogonal direct sum $N\F=N\F_Y\oplus NY$, where
$NY=TX|_Y/TY$ is the normal bundle of $Y$ in $X$.  Let $p$ be the
codimension of $\F_Y$ in $Y$ and $q$ the codimension of $Y$ in $X$.
Then $p+q=n$.  We form the orthonormal frame bundle $P_1$ of $N\F_Y$
and, assuming $Y$ to be co-orientable, the oriented orthonormal frame
bundle $P_2$ of $NY$.  The structure group of $P_1$ is
$K_1=\group{O}(p)$ and the structure group of $P_2$ is $K_2=\SO(q)$.
For every $x\in Y$ a pair consisting of a frame of $N_x\F_Y$ and a
frame of $N_xY$ gives a frame of $N_x\F$, so we have an embedding $j$
of the fibred product $P'=P_1\times_YP_2$ into $P$, which is
equivariant with respect to the embedding $K'=K_1\times K_2\to K$.  We
view $P'$ as a principal $K'$-bundle over $Y$.  Choose a
$K'$-invariant projection ${\pr}\colon\liek\to\liek'$; then by
Lemma~\ref{lemma;connection}\eqref{item;project} the form
$\theta'={\pr}\circ\theta\circ Tj$ is a connection on $P'$.  The
bundle $P_2$ is the quotient of $P'$ with respect to the $K_1$-action,
so by Lemma~\ref{lemma;connection}\eqref{item;extend} the form
$\theta'$ descends uniquely to a connection $\theta_Y$ on $NY$.  Since
$\theta$ is $\g$-invariant and $\F$-basic, $\theta'$ and $\theta_Y$
are $\g$-invariant and $\F_Y$-basic.  The following statement now
follows from Proposition~\ref{proposition;thom-form}.

\begin{proposition}\label{proposition;normal-thom}
Let $(X,\F,g)$ be a Riemannian foliated manifold equipped with an
isometric transverse action of a Lie algebra $\g$.  Let $Y$ be a
co-orientable $\g$-invariant submanifold of $X$.  Then the normal
bundle $NY$ possesses an invariant basic metric connection.  Hence
$NY$ has an equivariant basic Thom form, and hence the Thom
isomorphism theorem, Theorem~\ref{theorem;equivariant-basic-thom},
applies to $NY$.
\end{proposition}

\subsection{The Thom-Gysin sequences}

We now have all the ingredients for the equivariant basic version of
the Thom-Gysin sequences,
Propositions~\ref{proposition;thom-gysin-vector}
and~\ref{proposition;thom-gysin-submanifold}.

\begin{proposition}
\label{proposition;thom-gysin-vector-equivariant-basic}
Let $(M,\F,a)$ be foliated $\g$-manifold and $(E,\F_E,g_E)$ an
oriented $\g$-equivariant Riemannian foliated vector bundle over $M$.
Suppose that $E$ admits an invariant basic metric connection.  Let
$\tau_\g\in\Omega_{\g,\cv}^r(E,\F_E)$ be an equivariant basic Thom
form of $E$, $\eta_\g=\zeta^*\tau_\g\in\Omega_\g^r(M,\F)$ the
associated equivariant basic Euler form, and $\Eu_\g(E)=[\eta_\g]$ the
equivariant basic Euler class.  Let $SE$ be the sphere bundle of $E$
and let $\pi_{SE}=\pi|_{SE}\colon SE\to M$ be the projection.  Then we
have a long exact sequence
\[
\begin{tikzcd}[column sep=small]
\mbox{}\ar[r]&H_\g^{k-r}(M,\F)\ar[r,"\Eu_\g(E)\cup"]&[2em]
H_\g^k(M,\F)\ar[r,"\pi_{SE}^*"]&[1ex]
H_\g^k(SE,\F_{SE})\ar[r,"\pi_{SE,*}"]&[1em]H_\g^{k-r+1}(M,\F)\ar[r]&\mbox{}
\end{tikzcd}
\]
\end{proposition}

\begin{proof}[Sketch of proof]
The proof of Proposition~\ref{proposition;thom-gysin-vector} works in
the present context, relying on the equivariant basic Thom
isomorphism, Theorem~\ref{theorem;equivariant-basic-thom}, and on the
equivariant basic version of
Proposition~\ref{proposition;transgression}.
\end{proof}

For the submanifold version of this sequence we consider a manifold
$X$ equipped with a Riemannian foliation $(\F,g)$ and an isometric
transverse $\g$-action, and a closed $\g$-invariant submanifold
$i\colon Y\to X$ with normal bundle $N=i^*TX/TY$.
By~\cite[Proposition~3.3.4]{lin-sjamaar;riemannian} the existence of a
$\g$-equivariant foliate tubular neighbourhood embedding $i_N\colon
N\to X$ is guaranteed if $X$ is complete.  Suppose this to be the
case, and also that $Y$ is co-oriented.  Let $r$ be the codimension of
$Y$ and choose an equivariant basic Thom form
$\tau_\g(N)\in\Omega_{\cv,\g}^r(N,\F_N)$ of $N$, the existence of
which follows from Proposition~\ref{proposition;normal-thom}.  For
each $\beta\in\Omega_{\cv,\g}(N,\F_N)$ the form $(i_N^{-1})^*\beta$
extends by zero to an equivariant basic form on $X$, which we denote
by $i_{N,*}\beta$.  In the same way as~\eqref{equation;wrong-way} we
define the \emph{wrong-way homomorphism} to be the degree $r$ morphism
of complexes
\begin{equation}\label{equation;wrong-way-equivariant-basic}
i_*=i_{N,*}\circ\zeta_{N,*}\colon\Omega_\g(Y,\F_Y)[-r]\longto
\Omega_\g(X,\F),
\end{equation}
given by $i_*\alpha=i_{N,*}\bigl(\tau_\g(N)\wedge\pi_N^*\alpha\bigr)$.

\begin{proposition}
\label{proposition;thom-gysin-submanifold-equivariant-basic}
Let $(X,\F,g)$ be a complete Riemannian foliated manifold equipped
with an isometric transverse action of a Lie algebra $\g$.  Let $Y$ be
a closed co-orientable $\g$-invariant submanifold of codimension $r$.
Let $i\colon Y\to X$ and $j\colon X\backslash Y\to X$ be the
inclusions.  We have a long exact sequence
\[
\begin{tikzcd}[column sep=small]
\cdots\ar[r]&H_\g^{k-r}(Y,\F_Y)\ar[r,"i_*"]&H_\g^k(X,\F)\ar[r,"j^*"]&
H_\g^k(X\backslash Y,\F_{X\backslash
  Y})\ar[r]&H_\g^{k-r+1}(Y,\F_Y)\ar[r]&\cdots
\end{tikzcd}
\]
\end{proposition}

\begin{proof}[Sketch of proof]
We follow the proof of
Proposition~\ref{proposition;thom-gysin-submanifold}, using the
equivariant basic Thom isomorphism,
Theorem~\ref{theorem;equivariant-basic-thom}, the equivariant basic
version of Proposition~\ref{proposition;transgression}, and the
equivariant basic version of the excision lemma,
Lemma~\ref{lemma;excision}.  (The latter relies on the Mayer-Vietoris
principle for equivariant basic de Rham theory, for which
see~\cite[Proposition~3.3.7]{lin-sjamaar;riemannian}.)
\end{proof}

\appendix

\section{Cartan-Chern-Weil theory}\label{section;cartan-chern-weil}

This appendix is a summary of the algebraic principles of H. Cartan's
equivariant de Rham theory~\cite{cartan;algebre-transgression}.  We
follow the exposition
of~\cite{guillemin-sternberg;supersymmetry-equivariant},
\cite{alekseev-meinrenken;lie-chern-weil}
and~\cite{goertsches-toeben;equivariant-basic-riemannian}.  This
material is mostly standard, but we state a few results, notably
Theorems~\ref{theorem;principal} and~\ref{theorem;equivariant-ccw},
under weaker hypotheses than our references.

The action of a Lie algebra $\g$ on a manifold induces two actions on
differential forms, namely through Lie derivatives and through
contractions, and these actions intertwine in a distinctive way with
each other and with the exterior derivative.  These features are
abstracted in the notion of a $\g$-differential graded module.  An
example of a $\g$-differential graded module is the Weil algebra
$\W\g$, which models the de Rham complex of a universal bundle $E_G$,
where $G$ is a Lie group with Lie algebra $\g$.  A key feature of the
theory is a convenient criterion for the equivariant cohomology of a
$\g$-differential graded module to be isomorphic to its basic
cohomology, namely Theorem~\ref{theorem;principal} and its equivariant
analogue Theorem~\ref{theorem;equivariant-ccw}.  This criterion asks
whether the $\g$-differential graded module structure extends to a
compatible $\W\g$-module structure.  Such a $\W\g$-module structure
enables an algebraic version of Chern and Weil's connection-curvature
construction of characteristic forms.

In this section we place ourselves in the category of $\Z$-graded
vector spaces over a field $\FF$ of characteristic $0$.  An ungraded
object is considered as a graded object concentrated in degree $0$.
``Module'' means ``graded module'', ``dual'' means ``graded dual'',
``commutative'' means ``graded commutative'', ``derivation'' means
``graded derivation'', etc.  Tensor products are taken over $\FF$ and
equipped with the total grading.  Our complexes will be cochain
complexes.  We denote the translation functor on graded objects by
$[r]$.  So if $(C,d)$ is a cochain complex we have $C[k]^i=C^{i+k}$
and $d[k]=(-1)^kd$.  We will abbreviate ``differential graded module''
to ``dgm'' and ``differential graded algebra'' to ``dga''.

\subsection{The Lie algebra $\tilde\g$}\label{section;tildeg}

Let $\g$ be a finite-dimensional Lie algebra over $\FF$ (placed in
degree $0$).  Define the (graded) vector space $\tilde\g$ by
\[\tilde\g=\g[1]\oplus\g\oplus\FF[-1],\]
i.e.\ place copies of $\g$ in degrees $-1$ and $0$ and a copy of $\FF$
in degree $1$.  For $\xi\in\g$ denote the corresponding element of
$\g[1]$ by $\iota(\xi)$ and the corresponding element of $\g[0]$ by
$L(\xi)$.  Denote the basis element of $\FF[-1]$ corresponding to
$1\in\FF$ by $d$.  The conditions
\begin{equation}\label{equation;gstar}
\begin{aligned}
{[\iota(\xi),\iota(\eta)]}&=0,&[L(\xi),L(\eta)]&=L([\xi,\eta]),
&[d,d]&=0,\\
[L(\xi),d]&=0,&[\iota(\xi),d]&=L(\xi),&
[L(\xi),\iota(\eta)]&=\iota([\xi,\eta])
\end{aligned}
\end{equation}
for all $\xi$, $\eta\in\g$ determine a bilinear bracket on $\tilde\g$
which makes $\tilde\g$ a (graded) Lie algebra.  As for any (graded)
Lie algebra, we can talk about (graded) modules and algebras over
$\tilde\g$.

\subsection{$\g$-differential graded modules}
\label{section;gtilde-modules}

A \emph{$\g$-differential graded module} (abbreviated to
\emph{$\g$-dgm}) is a graded $\tilde\g$-module, in other words a
graded vector space $\M$ equipped with an endomorphism $d$ of degree
$1$ and, for each $\xi\in\g$, endomorphisms $\iota(\xi)$ of degree
$-1$ and $L(\xi)$ of degree $0$, which depend linearly on $\xi$ and
satisfy the commutation rules~\eqref{equation;gstar}.  In particular,
the rule $[L(\xi),L(\eta)]=L([\xi,\eta])$ means that a $\g$-dgm $\M$
is a $\g$-module, and the rule $d^2=\frac12[d,d]=0$ means that $\M$ is
a cochain complex of $\g$-representations.

Let $\M$ be a $\g$-dgm.  An element $m$ of $\M$ is
\emph{$\g$-invariant} if $L(\xi)\,m=0$ for all $\xi\in\g$,
\emph{$\g$-horizontal} if $\iota(\xi)\,m=0$ for all $\xi\in\g$, and
\emph{$\g$-basic} if it is $\g$-invariant and $\g$-horizontal.  We
denote by
\[
\M^\g,\qquad \M_\hhor{\g}=\M_\hor,\qquad
\M_\bbas{\g}=\M_\bas=\M^\g\cap\M_\hor
\]
the subspaces of $\M$ consisting of invariant, horizontal and basic
elements, respectively.  The subspaces $\M^\g$ and $\M_\bbas{\g}$ are
$\tilde\g$-submodules of $\M$.  The \emph{$\g$-basic cohomology} of
$\M$ is the vector space $H_\bbas{\g}(\M)=H(\M_\bbas{\g},d)$.  A
\emph{morphism of $\g$-dgm} $f\colon\M\to\M'$ is degree $0$ morphism
$f$ of $\tilde\g$-modules.  A \emph{homotopy of $\g$-dgm} between two
morphisms $f_0$, $f_1\colon\M\to\M'$ is a degree $-1$ linear map
$F\colon\M\to\M'[-1]$ satisfying
\[[\iota(\xi),F]=0,\qquad[L(\xi),F]=0,\qquad[d,F]=f_1-f_0\]
for all $\xi\in\g$.  Two homotopic morphisms $f_0$ and $f_1$ of
$\g$-dgm induce the same maps in cohomology $H(f_0)=H(f_1)\colon
H(\M)\to H(\M')$ as well as in basic cohomology
$H_\bas(f_0)=H_\bas(f_1)\colon H_\bas(\M)\to H_\bas(\M')$.

\subsection{$\g$-differential graded algebras}
\label{section;gtilde-algebras}

A \emph{$\g$-differential graded algebra} (abbreviated to
\emph{$\g$-dga}) is a graded $\tilde\g$-algebra, i.e.\ a
$\tilde\g$-module which is also an algebra, always assumed to be
unitary, associative and graded, on which the operators $d$, $L(\xi)$
and $\iota(\xi)$ act as graded derivations.  The basic complex of a
$\g$-dga $\A$ is a differential subalgebra of $\A$, so the basic
cohomology $\A$ is an $\FF$-algebra.

The algebra of $\FF$-linear endomorphisms $\E=\End(\M)$ of a $\g$-dgm
$\M$ is a $\g$-dga.  The basic subalgebra of $\E$ is
$\E_\bas=\End_{\tilde\g}(\M)$, the algebra of $\g$-dgm endomorphisms
of $\M$.

Let $\A$ be a $\g$-dga.  By an \emph{$\A$-module} we mean an
$\A$-module $\M$ \emph{in the category of $\g$-dgm}.  So $\M$ is also
equipped with a $\tilde\g$-module structure with the property that the
multiplication map $\A\otimes \M\to\M$ is $\tilde\g$-equivariant,
i.e.\ $\gamma(am)=\gamma(a)m+(-1)^{\abs{\gamma}\abs{a}}a\gamma(m)$ for
all homogeneous $\gamma\in\tilde\g$, $a\in\A$ and $m\in\M$.

We say that $\A$ is \emph{locally free} if it admits a
\emph{connection}, i.e.\ a linear map $\theta\colon\g^*\to\A^1$
satisfying
\[
\iota(\xi)(\theta(x))=\inner{\xi,x}\in\A^0\quad\text{and}\quad
L(\xi)(\theta(x))=-\theta(\ad^*(\xi)x)
\]
for all $\xi\in\g$ and $x\in\g^*$.  It is useful to reformulate this
notion as follows.  Let $\V\g$ be the vector space
$\tilde\g[1]^*=\g[2]^*\oplus\g[1]^*\oplus\FF[0]$.  For $x\in\g^*$
denote the corresponding element of $\g[1]^*$ by $\vartheta(x)$ and
the corresponding element of $\g[2]^*$ by $\dot\vartheta(x)$.  Denote
the degree $0$ element corresponding to $1\in\FF$ by $z$.  The rules
\begin{alignat}{3}
\label{equation;theta}
\iota(\xi)\vartheta(x)&=\inner{\xi,x},&
L(\xi)\vartheta(x)&=-\vartheta(\ad^*(\xi)x),&\quad
d\vartheta(x)&=\dot\vartheta(x),\\
\label{equation;dottheta}
\iota(\xi)\dot\vartheta(x)&=-\vartheta(\ad^*(\xi)x),&\quad
L(\xi)\dot\vartheta(x)&=-\dot\vartheta(\ad^*(\xi)x),&
d\dot\vartheta(x)&=0,\\
\label{equation;c}
\iota(\xi)z&=0,&L(\xi)z&=0,&dz&=0
\end{alignat}
for all $\xi\in\g$ and $x\in\g^*$ determine a structure of $\g$-dgm on
$\V\g$.  Here $\inner{\xi,x}\in\FF\subseteq\A^0$ denotes the dual
pairing.  The $\vartheta(x)$ are the \emph{connection elements} of
$\V\g$.  These rules ensure that a connection on $\A$ is equivalent to
a degree $0$ homomorphism of $\g$-dgm $\theta\colon\V\g\to\A$
satisfying $\theta(z)=1$.  We call a connection $\theta$
\emph{commutative} if the image $\theta(\V\g)$ generates a commutative
subalgebra of $\A$.

A typical example of a locally free $\g$-dga is $\A=\Omega(P)$, the de
Rham complex of a principal $G$-bundle $P$, where $G$ is any Lie group
with Lie algebra $\g$.  Typical examples of $\A$-modules are
$\M=\Omega_\cc(P)$, the compactly supported de Rham complex of $P$,
and $\M=\Omega(P,E)$, the de Rham complex with coefficients in an
equivariant flat vector bundle $E$ over $P$.

\subsection{The Weil algebra}\label{section;weil}

Let $S(\V\g)=S(\tilde\g[1]^*)$ be the (graded) symmetric algebra of
$\V\g$, equipped with the commutative $\g$-dga structure induced by
the $\g$-dgm structure on $\V\g$.  The \emph{Weil algebra} of $\g$ is
the commutative $\g$-dga $\W\g=S(\V\g)/(z-1)$, where $(z-1)$ is the
ideal generated by $z-1$.  The inclusion $\theta_{\text{\rm
    uni}}\colon\V\g\to\W\g$ is a connection on $\W\g$ called the
\emph{universal} or \emph{tautological connection}.  The Weil algebra
has the following universal property: every \emph{commutative}
connection $\theta$ on a $\g$-dga $\A$ is of the form
$\theta=c_\theta\circ\theta_{\text{\rm uni}}$ for a unique $\g$-dga
homomorphism $c_\theta\colon\W\g\to\A$, as in the diagram
\[
\begin{tikzcd}
\V\g\ar[d,"\theta_{\text{\rm uni}}"']\ar[r,"\theta"]&\A
\\
\W\g\ar[ur,dashed,"c_\theta"']
\end{tikzcd}
\]
(We have no need here of the noncommutative connections and
noncommutative Weil algebra
of~\cite{alekseev-meinrenken;lie-chern-weil}.)  We call $c_\theta$ the
\emph{characteristic homomorphism} of the connection.  Any two
connections on a principal bundle are homotopic.
See~\cite[Proposition 3.1]{alekseev-meinrenken;lie-chern-weil} for the
following algebraic counterpart of this fact.  We review the proof
(which is formally identical to the construction of homotopies in de
Rham theory; see Corollary~\ref{corollary;homotopy}), because we will
need the formula for the homotopy.

\begin{proposition}\label{proposition;connection}
Let $\theta_0$ and $\theta_1$ be commutative connections on a $\g$-dga
$\A$.  Suppose that $\theta_0$ and $\theta_1$ commute in the sense
that $[\theta_0(v_0),\theta_1(v_1)]=0$ for all $v_0$, $v_1\in\V\g$.
Then the characteristic homomorphisms $c_{\theta_0}$ and
$c_{\theta_1}$ are homotopic.
\end{proposition}

\begin{proof}
The \emph{graded line} is the Koszul differential algebra
$\group{S}=S(\FF[-2]\oplus\FF[-1])$ of $\FF$.  Let $s\in\FF[-2]$ and
$\dot{s}\in\FF[-1]$ be the two elements corresponding to the identity
$1\in\FF$.  The Koszul differential is given on generators by
$ds=\dot{s}$ and $d\dot{s}=0$.  For each $a\in\FF$ the
\emph{evaluation map} is the $\FF$-linear functional
$\ev_a\colon\group{S}\to\FF$ defined on basis elements by
$\ev_a(s^k)=a^k$ and $\ev(s^k\dot{s})=0$.  We often write $f(a)$
instead of $\ev_a(f)$.  The \emph{integral} is the $\FF$-linear
functional $J=\int_0^1\colon\group{S}\to\FF$ defined by $J(s^k)=0$ and
$J(s^k\dot{s})=1/(k+1)$.  These functionals are not homogeneous with
respect to the $\Z$-grading, but with respect to the $\Z/2\Z$-grading:
$\ev_a$ is even and $J$ is odd.  They satisfy the ``fundamental
theorem of calculus'' $J df=f(1)-f(0)$ for $f\in\group{S}$.  They
extend $\A$-linearly to functionals
\[
\ev_a=\ev_a\otimes\id_\A,\quad J=J\otimes\id_\A\colon\quad
\group{S}\otimes\A\to\A.
\]
The functionals $\ev_a$ are morphisms of $\g$-dgm and the functional
$J$ satisfies $[\iota(\xi),J]=[L(\xi),J]=0$ for all $\xi\in\g$.  The
fundamental theorem of calculus for $\alpha\in\group{S}\otimes\A$ says
that $J d\alpha=\alpha(1)-\alpha(0)-dJ\alpha$,
i.e. $[d,J]=\ev_1-\ev_0$.  Thus (the degree $-1$ component of) $J$ is
a homotopy from $\ev_0$ to $\ev_1$.  The $\g$-dga $\group{S}\otimes\A$
has a commutative connection $\theta$ defined by
$\theta(x)=(1-s)\otimes\theta_0+s\otimes\theta_1$ with corresponding
characteristic homomorphism $c_\theta\colon\W\g\to\group{S}\otimes\A$.
The operator $J\circ c_\theta\colon\W\g\to\A$ is then a homotopy from
$c_{\theta_0}$ to $c_{\theta_1}$.
\end{proof}

A few words about the structure of the Weil algebra.  The connection
elements $\vartheta(x)\in\g[1]^*\subseteq\W\g$ are of degree $1$.  The
map $\vartheta\colon\g^*\inj\V\g$ defined by $x\mapsto\vartheta(x)$
extends to an algebra homomorphism $\vartheta\colon\Lambda\g^*\cong
S(\g[1]^*)\inj\W\g$.  Similarly the map
$\dot\vartheta\colon\g^*\inj\V\g$ defined by
$x\mapsto\dot\vartheta(x)$ extends to an algebra homomorphism
$\dot\vartheta\colon S\g^*\cong S(\g[2]^*)\inj\W\g$.  The degree $2$
elements $\dot\vartheta(x)\in\g[2]^*\subseteq\W\g$ are not horizontal
(see~\eqref{equation;dottheta}), but the \emph{curvature elements}
$\mu(x)=\dot\vartheta(x)+\frac12\vartheta(\lambda(x))\in\W\g$ are,
where $\lambda\colon\g^*\to\Lambda^2\g^*$ is the map dual to the Lie
bracket.  The map $\mu\colon\g^*\to\W\g$ which sends $x$ to $\mu(x)$
extends to an algebra morphism $\mu\colon S\g^*\cong
S(\g[2]^*)\inj\W\g$.  See
e.g.~\cite[Ch.\ 3]{guillemin-sternberg;supersymmetry-equivariant}
or~\cite[\S\,3]{alekseev-meinrenken;lie-chern-weil} for the following
statement.

\begin{proposition}\phantomsection\label{proposition;weil}
\begin{enumerate}
\item\label{item;koszul}
The homomorphism $\vartheta\otimes\dot\vartheta\colon
S(\g[2]\oplus\g[1])^*\to\W\g$ is an isomorphism of algebras.  In terms
of the generators $\vartheta(x)$ and $\dot\vartheta(x)$ the
$\tilde\g$-structure equations of $\W\g$ are given by
\eqref{equation;theta}--\eqref{equation;dottheta}.  In particular, as
a complex $\W\g$ is isomorphic to the Koszul complex of $\g^*$ and is
therefore null-homotopic.
\item\label{item;connection-curvature}
The homomorphism $\vartheta\otimes\mu\colon
S(\g[2]\oplus\g[1])^*\to\W\g$ is also an isomorphism of algebras.  The
image of the homomorphism $\mu\colon S(\g[2]^*)\to\W\g$ is the
horizontal subalgebra $(\W\g)_\hor$.  In terms of the generators
$\vartheta(x)$ and $\mu(x)$ the $\tilde\g$-structure equations of
$\W\g$ read as follows:
\begin{alignat*}{3}
\iota(\xi)\vartheta(x)&=\inner{\xi,x},&\quad
L(\xi)\vartheta(x)&=-\vartheta(\ad^*(\xi)x),&\quad
d\vartheta(x)&=\frac12d_\CE\vartheta(x)+\mu(x),\\
\iota(\xi)\mu(x)&=0,&\quad L(\xi)\mu(x)&=-\mu(\ad^*(\xi)x),&\quad
d\mu(x)&=d_\CE\mu(x).
\end{alignat*}
\item\label{item;weil-basic}
The isomorphism $\mu\colon S(\g[2]^*)\to(\W\g)_\hor$ induces an
isomorphism $S(\g[2]^*)^\g\cong(\W\g)_\bas$.  The restriction of the
differential $d_{\W\g}$ to $(\W\g)_\bas$ is $0$, so $H_\bas(\W\g)\cong
S(\g[2]^*)^\g$.
\end{enumerate}
\end{proposition}

Here $d_\CE$ denotes the differential on $\W\g$ viewed as the
Cartan-Eilenberg complex $\Lambda\g^*\otimes S\g^*$ of the $\g$-module
$S\g^*$.  This differential sends $x\in\Lambda^1\g^*$ to
$-\lambda(x)\in\Lambda^2\g^*$ and $x\in S^1\g^*$ to the element of
$\Lambda^1\g^*\otimes S^1\g^*\cong\Hom(\g,S^1\g^*)$ given by
$d_\CE(x)(\xi)=L(\xi)(x)$.  The formulas for $d\vartheta$ and $d\mu$
are the \emph{Cartan-Bianchi identities}.

Let $\M$ be a $\g$-dgm.  Suppose that the algebra of $\FF$-linear
endomorphisms $\E=\End(\M)$ admits a commutative connection $\theta$.
Then the characteristic map $c_\theta\colon\W\g\to\E$ defines a
$\tilde\g$-linear multiplication law $\W\g\otimes\M\to\M$ on $\M$.  In
other words a $\W\g$-module structure on $\M$ is nothing but a
commutative connection $\theta$ on its endomorphism algebra.  With
this in mind we denote the multiplication law of a $\W\g$-module $\M$
by
\[
C_\theta\colon\W\g\otimes\M\longto\M\colon w\otimes m\longmapsto
c_\theta(w)m
\] 
and call it the \emph{Cartan-Chern-Weil homomorphism} of $\M$.  It is
a morphism of $\g$-dgm and therefore induces a morphism of complexes
\[(\W\g\otimes\M)_\bas\longto\M_\bas\]
and of graded vector spaces $H_\bas(\W\g\otimes\M)\to H_\bas(\M)$.

\subsection{Equivariant cohomology}\label{section;equivariant}

Let $\M$ be a $\g$-dgm.  The \emph{Weil complex} of $\M$ is
$\M_\g=(\W\g\otimes\M)_\bas$, the basic complex of the $\g$-dgm
$\W\g\otimes\M$, the differential of which we denote by $d_\g$.  The
cohomology of the Weil complex $H_\g(\M)=H(\M_\g)$ is the
\emph{equivariant cohomology} of $\M$.  A morphism of $\g$-dgm
$f\colon\M\to\M'$ induces a morphism in equivariant cohomology
$f_*\colon H_\g(\M)\to H_\g(\M')$.  Two homotopic morphisms induce the
same map in equivariant cohomology.

For the trivial $1$-dimensional $\g$-dgm $\M=\FF$ we have
$(\W\g\otimes\FF)_\bas=(\W\g)_\bas=S(\g[2]^*)^\g$, so
$H_\g(\FF)=S(\g[2]^*)^\g$.  For an arbitrary trivial $\g$-dgm $\M$ we
have $(\W\g\otimes\M)_\bas=(\W\g)_\bas\otimes\M$ with
$d_\g={\id}\otimes d_\M$, so $H_\g(\M)=S(\g[2]^*)^\g\otimes H(\M)$.

The map $\M\to\W\g\otimes\M$ sending $m$ to $1\otimes m$ is a morphism
of $\g$-dgm and therefore restricts to an injective morphism of
complexes $\M_\bas\to(\W\g\otimes\M)_\bas$, which induces a map
$H_\bas(\M)\to H_\g(\M)$.  The latter map is in general neither
injective nor surjective.

Suppose however that the $\g$-dgm structure on $\M$ extends to a
$\W\g$-module structure with multiplication map
$C_\theta\colon\W\g\otimes\M\to\M$ (where $\theta$ is the commutative
connection on the endomorphism algebra $\End(\M)$ that defines the
$\W\g$-module structure on $\M$).  For our purposes the following
statement, which is a variant of~\cite[Theorem
  5.2.1]{guillemin-sternberg;supersymmetry-equivariant}
and~\cite[Proposition 3.2]{alekseev-meinrenken;lie-chern-weil}, is the
main point of Cartan-Chern-Weil theory.  Note the absence of any
hypotheses on the Lie algebra $\g$.  In particular $\g$ need not be
reductive, nor is $\M$ required to be semisimple as a $\g$-module.

\begin{theorem}\label{theorem;principal}
Let $\M$ be a $\W\g$-module.  The Cartan-Chern-Weil map
$C_\theta\colon\W\g\otimes\M\to\M$ is a homotopy inverse of the
inclusion $j\colon\M\to\W\g\otimes\M$.  Therefore $j$ restricts to a
homotopy equivalence $\M_\bas\overset\simeq\longto\M_\g$ and induces
an isomorphism $H_\bas(\M)\overset\cong\longto H_\g(\M)$.
\end{theorem}

\begin{proof}
We have $C_\theta(j(m))=C_\theta(1\otimes m)=m$ for $m\in\M$, so
$C_\theta\circ j=\id_\M$.  Next we show that the map $h_1=j\circ
C_\theta$ is homotopic to the identity map $h_0=\id_{\W\g\otimes\M}$.
We have
\[
h_1(w\otimes m)=j(C_\theta(w\otimes m))=j(c_\theta(w)m)=1\otimes
c_\theta(w)m
\]
for $w\in\W\g$ and $m\in\M$.  Let $\D$ be the algebra
$\End(\W\g\otimes\M)$ and define $\g$-dga homomorphisms $f_0$,
$f_1\colon\W\g\to\D$ by
\[
f_0(v)(w\otimes m)=vw\otimes m,\qquad f_1(v)(w\otimes
m)=(-1)^{\abs{v}\abs{w}}w\otimes c_\theta(v)m
\]
for $v$, $w\in\W\g$ and $m\in\M$.  Then
\begin{equation}\label{equation;characteristic-homotopy}
\begin{aligned}
f_0(v)(1\otimes m)&=v\otimes m=h_0(v\otimes m),\\
f_1(v)(1\otimes m)&=1\otimes c_\theta(v)m=h_1(v\otimes m).
\end{aligned}
\end{equation}
We have $[f_0(v_0),f_1(v_1)]=0$ for all $v_0$, $v_1\in\W\g$, so the
maps $f_0$ and $f_1$ are the characteristic homomorphisms of a
commutating pair of connections $\theta_0$ and $\theta_1$ on $\D$.
Proposition~\ref{proposition;connection} gives us a homotopy
$F\colon\W\g\to\D[-1]$ satisfying $f_1-f_0=[d,F]$.  This homotopy is
given by $F(v)=\int_0^1c_\Theta(v)$ for $v\in\W\g$, and
$c_\Theta\colon\W\g\to\group{S}\otimes\D$ is the characteristic
homomorphism of the connection $\Theta(x)=(1-s)\otimes
\theta_0+s\otimes\theta_1$ on $\group{S}\otimes\D$.  Using
\eqref{equation;characteristic-homotopy} we obtain that
$h_1-h_0=[d,H]$, where $H\colon\W\g\otimes\M\to(\W\g\otimes\M)[-1]$ is
the homotopy given by
\[
H(w\otimes m)=F(w)(1\otimes m)=
\biggl(\int_0^1c_\Theta(w)\biggr)(1\otimes m)
\]
for $w\in\W\g$ and $m\in\M$.
\end{proof}

\subsection{Change of Lie algebra}\label{section;change}

A Lie algebra homomorphism $\h\to\g$ induces a homomorphism
$\tilde\h\to\tilde\g$ and hence a pullback functor from $\g$-dgm to
$\h$-dgm.  If $\M$ is a $\g$-dgm we have natural maps
\[
\M^\g\longto\M^\h,\quad\M_\hhor{\g}\longto\M_\hhor{\h},
\quad\M_\bbas{\g}\longto\M_\bbas{\h},\quad\M_\g\longto\M_\h,
\]
and hence maps in basic cohomology $H_\bbas{\g}(\M)\to
H_\bbas{\h}(\M)$ and in equivariant cohomology $H_\g(\M)\to H_\h(\M)$.
If $\M$ is an $\h$-dgm and the morphism $\h\to\g$ is surjective with
kernel $\liek$, then on the subcomplex $\M_\bbas{\liek}$ the
operations $L(\eta)$ and $\iota(\eta)$ for $\eta\in\liek$ are trivial,
so $\M_\bbas{\liek}$ is in a natural way a $\g$-dgm, and we have
\begin{equation}\label{equation;basic-restrict}
\M_\bbas{\h}=(\M_\bbas{\liek})_\bbas{\g}.
\end{equation}
This implies $H_\bbas\h(\M)=H_\bbas\g(\M_\bbas{\liek})$.  In the case
of a product of Lie algebras we have the following simple statement
about the Weil complex of $\M$.

\begin{lemma}\label{lemma;equivariant-ccw}
Let $\h=\liek\times\g$ be the product of two Lie algebras $\liek$ and
$\g$.
\begin{enumerate}
\item\label{item;weil-product}
$\W\h$ is isomorphic to $\W\liek\otimes\W\g$ as an $\h$-dga.
\item\label{item;module-product}
Let $\M$ be an $\h$-dgm.  The $\liek$-Weil complex $\M_{\liek}$ is a
$\g$-dgm and the $\h$-Weil complex $\M_\h$ is isomorphic to
$(\M_{\liek})_\g$ as an $\h$-dgm.  It follows that $H_\h(\M)\cong
H_\g(\M_{\liek})$.
\end{enumerate}
\end{lemma}

\begin{proof}
\eqref{item;weil-product}~The product of the universal connections
$\liek^*\to\W\liek$ and $\g^*\to\W\g$ is a connection
$\h^*\to\W\liek\otimes\W\g$ and therefore induces an $\h$-dga
homomorphism $\phi\colon\W\h\to\W\liek\otimes\W\g$.  The projections
$\h\to\liek$ and $\h\to\g$ induce two connections $\liek^*\to\W\h$ and
$\g^*\to\W\h$, and hence algebra maps
$\psi_{\liek}\colon\W\liek\to\W\h$ and $\psi_\g\W\g\to\W\h$.  The map
$\psi\colon\W\liek\otimes\W\g\to\W\h$ given by $\psi(v\otimes
w)=\psi_{\liek}(v)\psi_\g(w)$ is the inverse of $\phi$.

\eqref{item;module-product}~The algebra $\W\liek$ is trivial as a
$\g$-dgm and the $\tilde\g$-operations on $\M$ commute with the
$\tilde{\liek}$-operations, so
$\M_{\liek}=(\W\liek\otimes\M)_\bbas{\liek}$ is a $\g$-dgm.  It
follows from~\eqref{equation;basic-restrict} and
from~\eqref{item;weil-product} that we have an isomorphism of
complexes
\[
\M_\h=(\W\h\otimes\M)_\bbas\h=
\bigl((\W\g\otimes\W\liek\otimes\M)_\bbas{\liek}\bigr)_\bbas\g\cong
(\M_{\liek})_\g.
\]
Therefore $H_\h(\M)$ is isomorphic to $H_\g(\M_{\liek})$.
\end{proof}

Let $\M$ be an $\h$-dgm and $\E=\End(\M)$ its algebra of
endomorphisms.  Suppose that $\E$ is locally free as a $\liek$-dga, so
that it has a $\liek$-connection $\theta\colon\liek^*\to\E^1$.  We say
that the connection $\theta$ is \emph{$\g$-invariant} if
$\theta(y)\in\E^1$ is $\g$-invariant for all $y\in\liek^*$.
Similarly, $\theta$ is \emph{$\g$-horizontal} if $\theta(y)$ is
$\g$-horizontal for all $y\in\liek^*$; and $\theta$ is
\emph{$\g$-basic} if $\theta(y)$ is $\g$-basic for all $y\in\liek^*$.
Define
\[
\iota_\theta\colon\liek^*\longto\g[1]^*\otimes\E^0\cong
\Hom(\g[1],\E^0)
\]
by $\iota_\theta(y)(\xi)=-\iota(\xi)\theta(y)$ for $\xi\in\g$ and
$y\in\liek^*$.  Berline and Vergne's~\cite{berline-vergne;zeros}
\emph{$\g$-equivariant extension} of the connection $\theta$ is the
entity
\[
\theta_\g\colon\liek^*\longto
(\W\g\otimes\E)^1=\FF\otimes\E^1\oplus\g[1]^*\otimes\E^0
\]
defined by $\theta_\g=\theta+\iota_\theta$.  This terminology is
justified by the next theorem, which is a $\g$-equivariant version of
Theorem~\ref{theorem;principal}.  This theorem refines the earlier
results~\cite[\S\,4.6]{guillemin-sternberg;supersymmetry-equivariant}
and~\cite[Proposition
  3.9]{goertsches-toeben;equivariant-basic-riemannian} in two ways: it
is valid for arbitrary Lie algebras $\liek$ and $\g$, and the
cohomology isomorphism is given by an explicit homotopy equivalence.

\begin{theorem}\label{theorem;equivariant-ccw}
Let $\h=\liek\times\g$ be the product of two Lie algebras $\liek$ and
$\g$.  Let $\M$ be an $\h$-dgm.  Suppose that there exists a
$\g$-invariant connection $\theta\colon\liek^*\to\E^1$ on the
$\liek$-dga $\E=\End(\M)$.
\begin{enumerate}
\item\label{item;berline-vergne}
The $\g$-equivariant extension $\theta_\g=\theta+\iota_\theta$ is a
$\g$-basic connection on the $\liek$-dga $\E_\g=\W\g\otimes\E$.
\item\label{item;equivariant-ccw}
Suppose that the connection $\theta_\g$ is commutative.  Let
$c_{\g,\theta}\colon\W\liek\to\E_\g$ be the characteristic
homomorphism and
\[
C_{\g,\theta}\colon\W\h\otimes\M\cong\W\liek\otimes\W\g\otimes\M
\longto\W\g\otimes\M
\]
the Cartan-Chern-Weil map associated with $\theta_\g$.  Then
$C_{\g,\theta}$ is a homotopy inverse of the inclusion
$j\colon\W\g\otimes\M\to\W\h\otimes\M$.  It follows that $j$ induces a
homotopy equivalence $(\M_\bbas{\liek})_\g\overset\simeq\longto\M_\h$
and hence an isomorphism $H_\g(\M_\bbas{\liek})\overset\cong\longto
H_\h(\M)$.
\end{enumerate}
\end{theorem}
\glossary{cgtheta@$c_{\g,\theta}$, $\g$-equivariant characteristic
  homomorphism}
\glossary{Cgtheta@$C_{\g,\theta}$, $\g$-equivariant Cartan-Chern-Weil
  homomorphism}

\begin{proof}
\eqref{item;berline-vergne}~Let $\eta\in\liek$, $y\in\liek^*$ and
$\xi$, $\xi'\in\g$.  Then
\[
\iota(\eta)\iota_\theta(y)(\xi)=
-\iota(\eta)\iota(\xi)\theta(y)=\iota(\xi)\inner{\eta,y}=0,
\]
because the derivation $\iota(\xi)$ of $\E$ kills the scalar
$\inner{\eta,y}$.  It follows that
$\iota(\eta)\theta_\g(y)=\iota(\eta)\theta(y)=\inner{\eta,y}$.
Similarly,
\begin{align*}
L(\eta)\iota_\theta(\xi)&=-L(\eta)\iota(\xi)\theta(y)=
-\iota(\xi)L(\eta)\theta(y)=\iota(\xi)\theta(\ad^*(\eta)y)\\
&=-\iota_\theta(\ad^*(\eta)y)(\xi),
\end{align*}
because $[\xi,\eta]=0$.  Therefore
\begin{align*}
L(\eta)\theta_\g(y)&=L(\eta)\theta(y)+L(\eta)\iota_\theta(y)=
-\theta(\ad^*(\eta)y)-\iota_\theta(\ad^*(\eta)y)\\
&=-\theta_\g(\ad^*(\eta)y).
\end{align*}
This proves that $\theta_\g$ is a connection.  Next we have
\begin{align*}
(L(\xi)\iota_\theta(y))(\xi')&=
  L(\xi)(\iota_\theta(y)(\xi'))-\iota_\theta(y)(L(\xi)\xi')\\
&=-L(\xi)\iota(\xi')\theta(y)-\iota_\theta(y)([\xi,\xi'])\\
&=-\iota(\xi')L(\xi)\theta(y)-\iota([\xi,\xi'])\theta(y)
  +\iota([\xi,\xi'])\theta(y)\\
&=0,
\end{align*}
from which it follows that
$L(\xi)\theta_\g(y)=L(\xi)\theta_\g(y)+L(\xi)\iota_\theta(y)=0$.  So
$\theta_\g$ is $\g$-invariant.  The identity
$\iota(\xi)\theta_\g(y)=\iota(\xi)\theta(y)-\iota(\xi)\theta(y)=0$
shows that $\theta_\g$ is $\g$-horizontal.

\eqref{item;equivariant-ccw}~Because $\theta_\g$ is a
$\liek$-connection, the maps $c_{\g,\theta}$ and $C_{\g,\theta}$ are
$\liek$-dgm homomorphisms.  To see that they are $\h$-dgm
homomorphisms we must show that for all $v\in\W\liek$ the element
$c_{\g,\theta}(v)\in\E_\g$ is annihilated by $L(\xi)$ and $\iota(\xi)$
for all $\xi\in\g$.  It is enough to check this for the generators
$v=\vartheta(y)$ and $v=\dot\vartheta(y)$ of $\W\liek$ for
$y\in\liek^*$.  But the elements $\vartheta(y)$ and $\dot\vartheta(y)$
act on $\E_\g$ through multiplication by $\theta_\g(y)$ and
$d\theta_\g(y)$.  By \eqref{item;berline-vergne} the latter elements
are $\g$-basic, so they are annihilated by $L(\xi)$ and $\iota(\xi)$.
Thus $c_{\g,\theta}$ and $C_{\g,\theta}$ are $\h$-dgm homomorphisms.
It follows from Theorem~\ref{theorem;principal} that $C_{\g,\theta}$
is a $\liek$-homotopy inverse of the inclusion $j$.  According to the
proof of that theorem a $\liek$-homotopy from the identity map of
$\W\h\otimes\M$ to $j\circ C_{\g,\theta}$ is given by
\begin{equation}\label{equation;equivariant-homotopy}
H(v\otimes w\otimes m)=\biggl(\int_0^1c_\Theta(v)\biggr)(1\otimes
w\otimes m)
\end{equation}
for $v\in\W\liek$, $w\in\W\g$ and $m\in\M$.  The quantity
$c_\Theta\colon\W\liek\to\group{S}\otimes\D$ in this formula is the
characteristic homomorphism of the connection
$\Theta(x)=(1-s)\otimes\theta_0+s\otimes\theta_1$ on the algebra
$\group{S}\otimes\D$, where $\D=\End(\W\h\otimes\M)$.  Here $\theta_0$
and $\theta_1$ are the $\liek$-connections on $\D$ determined by the
following algebra homomorphisms $f_0$, $f_1\colon\W\liek\to\D$:
\begin{align*}
f_0(u)(v\otimes w\otimes m)&=uv\otimes w\otimes m,\\
f_1(u)(v\otimes w\otimes m)&=(-1)^{\abs{u}\abs{v}}v\otimes
c_{\g,\theta}(u)(w\otimes m)
\end{align*}
for $u$, $v\in\W\liek$, $w\in\W\g$ and $m\in\M$.  The connection
$\theta_0$ is the universal $\liek$-connection $\liek^*\to\W\liek$
pulled back to $\W\h\otimes\M=\W\liek\otimes\W\g\otimes\M$ and so is
$\g$-basic.  The connection $\theta_1$ is $\g$-basic by
\eqref{item;berline-vergne}.  Therefore $\Theta$ is $\g$-basic.  The
graded line $\group{S}$ is a trivial $\g$-dgm and the functional
$\int_0^1\colon\group{S}\otimes\D\to\D$ commutes with the operations
$\iota(\xi)$ and $L(\xi)$ for $\xi\in\g$.  It follows that the map $H$
defined in~\eqref{equation;equivariant-homotopy} commutes with
$\iota(\xi)$ and $L(\xi)$.  We conclude that the $\liek$-homotopy $H$
is an $\h$-homotopy.
\end{proof}

Under the assumptions of the theorem we call the map $C_{\g,\theta}$
the \emph{$\g$-equivariant Cartan-Chern-Weil homomorphism} defined by
the connection $\theta$.  The hypothesis that $\theta_\g$ should be
commutative is satisfied in all our applications, but one can extend
Theorem~\ref{theorem;equivariant-ccw} to noncommutative connections by
resorting to the noncommutative Weil algebra
of~\cite{alekseev-meinrenken;lie-chern-weil}.

In the special case where $\M=\A$ is a commutative $\h$-dga equipped
with a $\g$-invariant $\liek$-connection $\theta$, we have a
$\g$-equivariant characteristic homomorphism
\[
c_{\g,\theta}\colon\W\liek\to\W\g\otimes\A.
\]
Taking $\h$-basics and applying
Lemma~\ref{lemma;equivariant-ccw}\eqref{item;module-product} gives an
algebra homomorphism
\begin{equation}\label{equation;characteristic}
c_{\g,\theta}\colon S(\liek[2]^*)^{\liek}\to(\A_\bbas{\liek})_\g.
\end{equation}
Taking cohomology then gives an algebra homomorphism that is
independent of the connection,
\begin{equation}\label{equation;characteristic-cohomology}
c_\g\colon S(\liek[2]^*)^{\liek}\to H_\g(\A_\bbas{\liek}).
\end{equation}
The maps~\eqref{equation;characteristic} and
\eqref{equation;characteristic-cohomology} are also referred to as
$\g$-equivariant characteristic homomorphisms.

\subsection{Group versus algebra}

Let $\FF=\R$ or $\C$ and let $G$ be a Lie group with Lie algebra $\g$.
Let $\M$ be a $\g$-dgm and suppose that $G$ acts linearly on $\M$ in a
way which is compatible with the $\tilde\g$-module structure in the
sense
of~\cite[\S\,3.2.1]{guillemin-sternberg;supersymmetry-equivariant}.
(For this to make sense we must assume either that the $G$-module $\M$
is locally finite, or that $\M$ has a complete Hausdorff locally
convex topology in which the function $G\to M$ defined by $g\mapsto
gm$ is smooth for every $m\in\M$.)  The \emph{$G$-basic complex} of
$\M$ is the set of $\g$-horizontal $G$-fixed vectors,
$\M_\bbas{G}=\M_\hhor{\g}\cap\M^G$.  The \emph{$G$-basic cohomology}
is $H_\bbas{G}(\M)=H(\M_\bbas{G})$.  The \emph{$G$-equivariant
  cohomology} is $H_G(\M)=H((\W\g\otimes\M)_\bbas{G})$.

\begin{lemma}\label{lemma;component}
Suppose that $G$ has finitely many components.  Let $\Pi=\pi_0(G)$ be
the component group of $G$ and let $\M$ be a $\g$-dgm with a
compatible $G$-action.  Then $\M_\bbas{G}=(\M_\bbas{\g})^\Pi$,
$H_\bbas{G}(\M)\cong H_\bbas{\g}(\M)^\Pi$, and $H_G(\M)\cong
H_\g(\M)^\Pi$.
\end{lemma}

\begin{proof}
The equality $\M_\bbas{G}=(\M_\bbas{\g})^\Pi$ follows from the
compatibility of the $G$-action.  The group $\Pi$ being finite, a
simple averaging argument shows that
$H\bigl((\M_\bbas{\g})^\Pi\bigr)\cong H(\M_\bbas{\g})^\Pi$, which
proves the second assertion.  The last assertion follows from the
second applied to the module $\W\g\otimes\M$.
\end{proof}

\section{Fibre integration}\label{section;fibre}

\numberwithin{equation}{section}

By integrating a differential form along the fibres of a submersion
one obtains a form of lower degree on the base manifold.  We review a
few properties of this useful process.  Let $\pi\colon E\to B$ be a
smooth oriented locally trivial fibre bundle with fibre $F$.  The base
$B$ and the fibre $F$ are allowed to have a boundary.  Let $r$ be the
dimension of $F$.  \emph{Fibre integration} or \emph{pushforward} is a
map
\[\pi_*\colon\Omega_\cv(E)[r]\longto\Omega(B),\]
which is uniquely determined by the requirement that it satisfy the
\emph{projection formula}
\begin{equation}\label{equation;projection}
\pi_*(\beta\wedge\pi^*\alpha)=\pi_*\beta\wedge\alpha
\end{equation}
for all $\alpha\in\Omega(B)$ and $\beta\in\Omega_\cv(E)$.  The
subscript ``$\cv$'' indicates vertically compactly supported forms;
see \S\,\ref{section;preliminary}.  For existence and uniqueness of
$\pi_*$ see e.g.~\cite[\S\,13]%
{derham;varietes-differentiables-formes-courants;;1973},
\cite[\S\,6]{bott-tu;differential-forms},
or~\cite[\S\,10.1]{guillemin-sternberg;supersymmetry-equivariant}.  We
adopt the convention~\eqref{equation;projection}, which agrees
with~\cite{derham;varietes-differentiables-formes-courants;;1973}, but
differs by a sign from~\cite{bott-tu;differential-forms}
and~\cite{guillemin-sternberg;supersymmetry-equivariant}, so as to
comply with the Koszul sign rule.  The projection formula is
equivalent to
$\pi_*(\pi^*\alpha\wedge\beta)=(-1)^{rk}\alpha\wedge\pi_*\beta$, where
$k$ is the degree of $\alpha$.  In other words
\[[\pi_*,l(\alpha)]=0,\]
where $l(\alpha)$ denotes left multiplication by $\alpha$ on
$\Omega(B)$ and left multiplication by $\pi^*\alpha$ on
$\Omega_\cv(E)$, and $[\pi_*,l(\alpha)]$ denotes the (graded)
commutator of $\pi_*$ and $l(\alpha)$.  Thus $\pi_*$ is a degree $-r$
morphism of left $\Omega(B)$-modules.
\glossary{p.istar@$\pi_*$, fibre integration}

The proof of the next lemma is a routine verification based on the
projection formula.  In part~\eqref{item;stokes} we denote by
$E^\partial$ the manifold $E^\partial=\bigcup_{b\in B}\partial E_b$,
which is a bundle over $B$ with fibre $\partial F$ and projection
$\pi^\partial=\pi|_{E^\partial}$.  (If the base $B$ has no boundary,
then $E^\partial=\partial E$.)

\begin{lemma}\phantomsection\label{lemma;fibre}
\begin{enumerate}
\item\label{item;fibre-derivation}
For every pair of $\pi$-related vector fields $v\in\lie{X}(B)$ and
$w\in\lie{X}(E)$ we have $L(v)\circ\pi_*=\pi_*\circ L(w)$ and
$\iota(v)\circ\pi_*=(-1)^r\pi_*\circ\iota(w)$.
\item\label{item;stokes}
Let $\pi^\partial_*\colon\Omega_\cv(E^\partial)[r-1]\to \Omega(B)$ be
the fibre integral for $\pi^\partial\colon E^\partial\to B$.  Then
$[\pi_*,d]=\pi_*^\partial$.
\item\label{item;pullback}
A pullback diagram of oriented fibre bundles
\[
\begin{tikzcd}
E_2\ar[r,"f_E"]\ar[d,"\pi_2"']&E_1\ar[d,"\pi_1"]
\\
B_2\ar[r,"f_B"]&B_1
\end{tikzcd}
\]
induces a commutative diagram
\[
\begin{tikzcd}
\Omega_\cv(E_2)[r]\ar[d,"\pi_{2,*}"']&
\Omega_\cv(E_1)[r]\ar[l,"f_E^*"']\ar[d,"\pi_{1,*}"]
\\
\Omega(B_2)&\Omega(B_1)\ar[l,"f_B^*"'].
\end{tikzcd}
\]
\end{enumerate}
\end{lemma}

We spell out some consequences of part~\eqref{item;stokes}.  First, if
the fibre has no boundary, fibre integration commutes with the
differential.  See \S\,\ref{section;gtilde-modules} for the definition
of $\g$-differential graded modules, and their morphisms and
homotopies.

\begin{corollary}\label{corollary;cochain}
If $\partial F=\emptyset$, then
$\pi_*\colon\Omega_\cv(E)[r]\to\Omega^*(B)$ is a morphism of cochain
complexes of degree $-r$.  If in addition the bundle $E\to B$ is
equivariant with respect to the action of a Lie algebra $\g$ on $E$
and $B$, then $\pi_*$ is a degree $-r$ morphism of $\g$-differential
graded modules.
\end{corollary}

If on the other hand $E$ is the cylinder $[0,1]\times B$, then the
bundle $E^\partial$ consists of two copies $E_0$ and $E_1$ of $E$, and
$\pi^\partial_*\beta=\beta|_{E_1}-\beta|_{E_0}$.

\begin{corollary}\label{corollary;cylinder}
Let $E=[0,1]\times B$.  Let $i_t\colon B\to E$ be the embedding
$i_t(b)=(t,b)$ and $i_t^*\colon\Omega(E)\to\Omega(B)$ the induced
morphism.  Then we have the cylinder formula
$i_1^*-i_0^*=d\pi_*+\pi_*d$.  If a Lie algebra $\g$ acts on $B$, then
$\pi_*$ is a homotopy of $\g$-differential graded modules.
\end{corollary}

Substituting a homotopy of maps into the cylinder formula gives the
homotopy formula.

\begin{corollary}\label{corollary;homotopy}
Let $h\colon[0,1]\times B_1\to B_2$ be a smooth homotopy.  Let
$h_t(b)=h(t,b)$ and let $\pi\colon[0,1]\times B_1\to B_1$ be the
projection.  Then $\pi_*h^*\colon\Omega(B_2)[1]\to\Omega(B_1)$ is a
cochain homotopy: $h_1^*-h_0^*=d\pi_*h^*+\pi_*h^*d$.  If $h$ is
equivariant with respect to a Lie algebra $\g$ acting on $B_1$ and
$B_2$, then $\pi_*h^*$ is a homotopy of $\g$-differential graded
modules.
\end{corollary}


\bibliographystyle{amsplain}

\bibliography{hamilton}


\end{document}